\documentclass[11pt, a4paper]{article}
\usepackage[utf8]{inputenc}
\usepackage{amsmath,amssymb,amsfonts,dsfont,stmaryrd}
\usepackage{amsthm}
\usepackage{mathtools}
\usepackage{verbatim}
\usepackage{a4wide}
\usepackage{epsfig,epic,eepic,graphicx}
\usepackage[mathcal]{euscript}
\usepackage{import}
\usepackage{xifthen}
\usepackage[normalem]{ulem}

\usepackage{marginnote} 
\reversemarginpar

\usepackage[usenames,dvipsnames]{xcolor}

\usepackage[unicode, bookmarks, colorlinks, breaklinks, pagebackref,
]{hyperref}  
\hypersetup{linkcolor=OliveGreen,citecolor=OliveGreen,filecolor=black,urlcolor=Magenta}

\usepackage[capitalize,nameinlink,noabbrev]{cleveref}

\usepackage{cite}

\usepackage{enumitem}

\newtheorem{theorem}{Theorem}
\newtheorem{lemma}{Lemma}[section]
\newtheorem{proposition}[lemma]{Proposition}

\newtheorem{definition}{Definition}

\newtheorem{remark}[lemma]{Remark}

\numberwithin{equation}{section}
\makeatletter

\@addtoreset{equation}{section}
\makeatother

\newcommand{\NN}{{\mathbb N}}
\newcommand{\ZZ}{{\mathbb Z}}

\newcommand{\RR}{{\mathbb R}}

\newcommand{\MM}{{\mathbb M}}

\newcommand{\causs}{\vec{\sigma}} 
\newcommand{\maxss}{\vec{\tau}} 

\newcommand{\emE}{\mathbf{E}}

\newcommand{\emD}{\mathbf{D}}

\newcommand{\ynote}[1]{{\color{blue}#1}}

\newcommand{\DD}{{\mathbb{D}}}

\newcommand{\email}[1]{\href{mailto:#1}{\texttt{#1}}}
\newcommand{\bfa}{\mathbf{a}}
\newcommand{\bfb}{\mathbf{b}}
\newcommand{\bfc}{\mathbf{c}}
\newcommand{\bfC}{\mathbf{C}}
\newcommand{\bfd}{\mathbf{d}}
\newcommand{\bfB}{\mathbf{B}}

\newcommand{\bfp}{\mathbf{p}}

\newcommand{\bg}{\mathbf{g}}
\newcommand{\ff}{\mathbf{f}}
\newcommand{\nn}{\mathbf{n}}
\newcommand{\ee}{\mathrm{e}}

\newcommand{\uu}{\mathbf{u}}

\newcommand{\vv}{\mathbf{v}}
\newcommand{\ww}{\mathbf{w}}

\newcommand{\yy}{y}

\renewcommand\vec[1]{\boldsymbol #1} 

\newcommand{\charfn}{\mathds{1}}

\newcommand{\per}{\mathrm{per}}
\newcommand{\sym}{\mathrm{sym}}
\newcommand{\mat}{\mathrm{mat}}
\newcommand{\hmeas}{\mathcal{H}^{d-1}}

\newcommand{\bfU}{\mathbf{U}}

\newcommand{\abs}[1]{\left\lvert #1\right\rvert}

\newcommand{\norm}[1]{\left\lVert #1\right\rVert}

\newcommand{\cv}[1][]{%
\ifthenelse{\isempty{#1}}{\xrightarrow[\hphantom{~2~}]{}}{\xrightarrow[\hphantom{~2~}]{#1}}%
}
\newcommand{\wcv}[1][]{%
\ifthenelse{\isempty{#1}}{\xrightharpoonup[\hphantom{~2~}]{}}{\xrightharpoonup[\hphantom{~2~}]{#1}}%
}

\newcommand{\calD}{\mathcal{D}}

\newcommand{\calL}{\mathcal{L}}
\newcommand{\calN}{\mathcal{N}}

\newcommand{\calM}{\mathcal{M}}
\newcommand{\fraM}{\mathfrak{M}}

\newcommand{\loc}{\mathrm{loc}}

\DeclareMathOperator{\diam}{diam}

\DeclareMathOperator{\di}{d\!}

\DeclareMathOperator{\Div}{{div}}

\def\XXint#1#2#3{{\setbox0=\hbox{$#1{#2#3}{\int}$ }
\vcenter{\hbox{$#2#3$ }}\kern-.6\wd0}}

\DeclarePairedDelimiterX\Set[1]\{\}{%
  #1%
}
\allowdisplaybreaks

\begin{document}


\title{Explicit corrector in homogenization of monotone operators and its application to nonlinear dielectric elastomer composites}

\author{ Thuyen Dang\footnote{Department of Statistics/Committee on
    Computational and Applied Mathematics, University of Chicago, 5747
    S. Ellis Avenue, Chicago, Illinois 60637, USA
    (\email{thuyend@uchicago.edu}).}
  \and
  Yuliya Gorb\footnote{National Science Foundation, 2415 Eisenhower Avenue,
    Alexandria, Virginia 22314, USA (\email{ygorb@nsf.gov}). }
  \and
  Silvia Jim\'{e}nez Bola\~{n}os\footnote{Department of Mathematics,
    Colgate University, 13 Oak Drive, Hamilton, New York 13346, USA
    (\email{sjimenez@colgate.edu}). } }

\date{}

\maketitle

\begin{abstract}
  This paper concerns the rigorous periodic homogenization for a
  weakly coupled electroelastic system of a nonlinear electrostatic
  equation with an elastic equation enriched with
  electrostriction. Such coupling is employed to describe dielectric
  elastomers or deformable (elastic) dielectrics. It is shown that the effective response of
  the system consists of a homogeneous dielectric elastomer described
  by a nonlinear weakly coupled system of PDEs whose coefficients
  depend on the coefficients of the original heterogeneous material, the geometry of the composite, and the periodicity of the original
  microstructure. The approach developed here for this {\it nonlinear}
  problem allows us to obtain an
    explicit corrector result for the homogenization of monotone
    operators with minimal regularity assumptions. Two $L^p-$gradient
    estimates for elastic systems with discontinuous coefficients are
    also obtained.
\end{abstract}


\section{Introduction}
\label{sec:introduction}

In recent years there has been a growing interest towards a class of materials known as dielectric elastomers that can exhibit coupled electrical and mechanical behavior, see cf. \cite{carpiDielectricElastomersElectromechanical2011}.
A unique property possessed by such materials, known as {\it electrostriction},  which means they can respond to an external electric field by changing their size and shape, makes them appealing for various practical applications, e.g. 
soft robotics, artificial muscles, active camouflage, haptic devices,
energy harvesting, and others, see e.g. \cite{carpiDielectricElastomersElectromechanical2011} and \cite{erturkPiezoelectricEnergyHarvesting2011}.
Homogenization theory can be used to guide the design of dielectric elastomers with enhanced electromechanical couplings, e.g. \cite{hakimisiboniDielectricElastomerComposites2014}, and this paper is devoted to the rigorous periodic homogenization of such a coupling.


The governing equations describing the system under consideration consist of a nonlinear (scalar) electrostatic equation in the presence of a bounded free body charge weakly coupled with an elastic (vectorial) equation that involves an {\it electrostriction} term. Here, {\it weakly} coupling means that the elastic displacement does not enter the electrostatic equation. The PDE system is posed on a heterogeneous bounded domain with {\it periodic microstructure} of size $0< \varepsilon \ll 1$. For simplicity, we focus on Dirichlet boundary conditions only. 
The goals of this paper are twofold. First, it aims at determining the {\it macroscopic} or {\it effective} behavior of the considered periodic composite under the assumption of scale separation. This amounts to developing an asymptotic analysis of the limiting response of the given PDE system as  $\varepsilon$, the size of the microstructure, goes to zero. Second, since many necessary facts (e.g. the regularity of the solution to the original fine-scale problem) are not readily available, this paper presents a number of stand-alone results that could be utilized in future contributions to the topic of periodic homogenization for nonlinear electrostatic and/or elastic composite materials under minimal regularity assumptions of the original system.

The rigorous justification of a model for dielectric elastomers and the
derivation of its effective system, using the mathematical theory of
homogenization, were carried out in
\cite{tianDielectricElastomerComposites2012,francfortEnhancementElastodielectricsHomogenization2021}
and references therein. To carry out the homogenization of their system, the
authors of \cite{tianDielectricElastomerComposites2012} made a
strict integrability assumption, requiring that the solution of the
electrostatic equation belongs to a class of least $W^{1,4}$-functions. The results in
\cite{francfortEnhancementElastodielectricsHomogenization2021} showed
that if the coefficients of the electrostatic equation are piecewise
H\"{o}lder continuous, then indeed its solution belongs to $W^{1,p}$,
for any $p \in [1,\infty)$. Later, in
\cite{dangGlobalGradientEstimate2022}, this result was extended to
the case $p =\infty$. A similar model for the case of magnetic
suspensions was investigated in
\cite{dangHomogenizationNondiluteSuspension2021,dangGlobalGradientEstimate2022}.

In the contributions cited above
\cite{tianDielectricElastomerComposites2012,francfortEnhancementElastodielectricsHomogenization2021}, the materials studied were {\it linear}, i.e., the constitutive
relationship between the electric field $\emE$ and the electric
displacement $\emD$ was assumed linear. However, when this
relationship is {\it nonlinear}, e.g., as in the case of {\it
  ferroelectric materials}, a new model, as well as a new approach,
are required to obtain the corresponding homogenized response. In
this paper, we consider $\emE$ and $\emD$ satisfying a nonlinear
constitutive relation that yields a {\it nonlinear divergence
  equation}. Therefore, the improved gradient estimates, obtained by
the compactness method used in
\cite{francfortEnhancementElastodielectricsHomogenization2021,dangGlobalGradientEstimate2022},
are no longer available. Instead, in this paper, we derive a new
approach that does not require such estimates or the renormalization
framework of
\cite{gaudielloHomogenizationBrushProblem2017,muratHomogenizationRenormalizedSolutions1991},
which are the typical techniques to deal with problems of low
regularity source terms. To gain additional regularity of the
solution, we apply regularity theory to not only the fine-scale or the
effective systems but also to the two-scale homogenized one. Two-scale
convergence acts as an ``intermediate" topology between weak and
strong convergence that allows canceling the dependence on the size of
microstructure, thus providing a system with rather nice
coefficients, besides the effective one. This idea enables us to relax
several regularity assumptions needed before in
\cite{francfortEnhancementElastodielectricsHomogenization2021,tianDielectricElastomerComposites2012,dangGlobalGradientEstimate2022},
and also to extend the cited works to
nonlinear cases.

To implement this idea for the current problem, several ingredients
are needed: a result from the nonlinear Cald\'{e}ron-Zygmund theory
\cite{phucGlobalIntegralGradient2014} (see also
\cite{kuusiGuideNonlinearPotential2014,mingioneGradientEstimatesDuality2010}),
an estimate for elliptic systems
\cite{liEstimatesEllipticSystems2003}, the existence of the
(generalized) Green's function \cite{conlonGreenFunctionElliptic2017},
the SOLA technique (existence of the solution by limit of
approximations \cite{boccardoNonlinearEllipticParabolic1989}), results
from the theory of two-scale convergence (for $L^p$ and $\mathrm{BV}$
functions)
\cite{allaireHomogenizationTwoscaleConvergence1992,nguetsengGeneralConvergenceResult1989,amarTwoscaleConvergenceHomogenization1998,ferreiraReiteratedHomogenizationBV2012},
an interpolation theorem \cite{grafakosModernFourierAnalysis2014}, a
corrector result in homogenization of monotone operators \cite{dalmasoCorrectorsHomogenizationMonotone1990}, and a duality argument.

Along with the sought-after homogenized response, in this paper we
obtain three stand-alone results, namely: (i) \cref{sec:main-results-1},
  which provides an explicit first-order corrector for the nonlinear electrostatic problem
  with minimal regularity assumptions, while (ii) \cref{sec:appendix-1} and  (iii)
  \cref{sec:appendix} provide $L^p-$gradient estimates for the elasticity
  system via an interpolation argument.
Besides the fact that our approach is built for the {\it nonlinear} problem with {\it minimal regularity assumptions}, combined with ideas from
\cite{cherdantsevHighContrastHomogenisation2017,davoliHomogenizationHighcontrastMedia2023,balHomogenizationHydrodynamicTransport2021} it could also be extended to the {\it high-contrast} case, for which the minimal and maximal values of coefficients of the underlying PDEs are vastly different. This case will be reported in future publications by the authors elsewhere.

This paper is organized as follows. In
\cref{sec:formulation}, the main notations are introduced and
the formulation of the fine-scale problem is discussed. Our main
result is stated in \cref{sec:main-result}, the conclusions
are given in \cref{sec:conclusions}, and
auxiliary facts are discussed in \cref{sec:an-appendix}.

\section{Formulation}
\label{sec:formulation}
\subsection{Notation and definitions}
\label{ss:setup}

Throughout this paper, the scalar-valued functions, such as the
pressure $p$, are written in usual typefaces, while vector-valued or
tensor-valued functions, such as the displacement $\uu$ and the
stress tensor $\causs$, are written in bold.
Sequences are indexed by
  superscripts ($\phi^i$), while elements of vectors or tensors
are indexed by numeric subscripts ($x_i$). Finally, the Einstein
summation convention is used whenever applicable; $\delta_{ij}$ is the Kronecker delta, and $\epsilon_{ijk}$ is the permutation symbol. The letter $C$ represents a constant, independent of $\varepsilon$, that can represent different values from line to line. 

Consider $\Omega \subset \RR^d$, for $d \ge 2$, a simply connected and
bounded domain of class $C^{1,1}$, occupied by a deformable
electroelastic material, and let $Y\coloneqq \left[ -\frac{1}{2}, \frac{1}{2} \right]^d$ be the unit cell
in $\RR^d$. The unit cell $Y$ is decomposed into:
$$Y=Y_s\cup Y_f \cup \Gamma,$$
where $Y_s$, representing the inclusion, and $Y_f$,
representing the matrix, are open sets in
$\mathbb{R}^d$, and $\Gamma$ is the closed $C^{1,1}$ interface that
separates them.

Let $i = (i_1, \ldots, i_d) \in \ZZ^d$ be a vector of indices and $\{\ee^1, \ldots, \ee^d\}$ be the canonical basis of $\RR^d$. 
For a fixed small $\varepsilon > 0,$ we define the dilated sets: 
\begin{align*}
  Y^{\varepsilon}
  \coloneqq \varepsilon Y, ~~
  Y^\varepsilon_i 
  \coloneqq \varepsilon (Y + i),~~
  Y^\varepsilon_{i,s}
  \coloneqq \varepsilon (Y_s + i),~~
  Y^\varepsilon_{i,f}
  \coloneqq \varepsilon (Y_f + i),~~
  \Gamma^\varepsilon_i 
  \coloneqq \partial Y^\varepsilon_{i,s}.
\end{align*}
Typically, in homogenization theory, the positive number $\varepsilon
\ll 1$ is referred to as the {\it size of the microstructure}. The
effective or homogenized response of the given suspension
corresponds to the case $\varepsilon=0$, whose  derivation and justification is the main focus of this paper. 

We denote by $\nn_i,~\nn_{\Gamma}$ and $\nn_{\partial \Omega}$ the unit normal vectors to $\Gamma^{\varepsilon}_{i}$ pointing outward $Y^\varepsilon_{i,s}$, on $\Gamma$ pointing outward $Y_{s}$ and on $\partial \Omega$ pointing outward, respectively; and also, we denote by $\di \hmeas$ the $(d-1)$-dimensional Hausdorff measure.
In addition, we define the sets:
\begin{align}
  \label{eq:116}
  \begin{split}
    I^{\varepsilon} 
    &\coloneqq \{ 
    i \in \ZZ^d \colon Y^\varepsilon_i \subset \Omega
    \},~~
    \Omega_s^{\varepsilon} 
    \coloneqq \bigcup_{i\in I^\varepsilon}
Y_{i,s}^{\varepsilon},~~
    \Omega_f^{\varepsilon} 
    \coloneqq \Omega \setminus \Omega_s^{\varepsilon},~~
    \Gamma^\varepsilon 
    \coloneqq \bigcup_{i \in I^\varepsilon} \Gamma^\varepsilon_i,\\
  J^{\varepsilon}
  &\coloneqq
  \left\{
  j \in \ZZ^d\colon Y_j^{\varepsilon} \cap \left( \RR^d \setminus \Omega \right) \ne
  \varnothing \text{ and } Y_j^{\varepsilon} \cap \Omega \ne \varnothing
  \right\},\\
  Z_i^{\varepsilon}
  &\coloneqq
  \begin{cases}
    Y_i^{\varepsilon} &\text{ if } i \in I^{\varepsilon},\\
    Y_i^{\varepsilon} \cap \Omega &\text{ if } i \in J^{\varepsilon},
  \end{cases}
  \end{split}
\end{align}
see \cref{fig:1}.

\begin{figure}[ht]
\centering
\def\svgwidth{0.5\columnwidth}
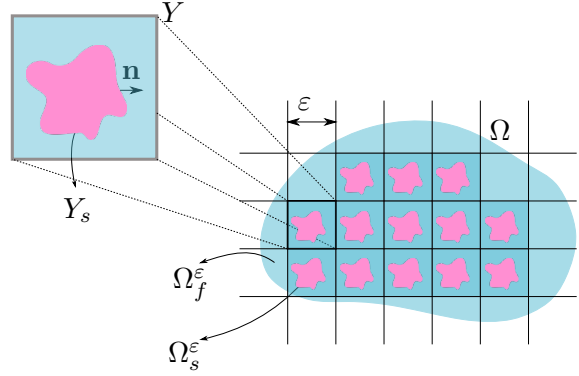
\caption{Reference cell $Y$ and domain $\Omega$.}
\label{fig:1}
\end{figure}

The following spaces are used throughout
this paper.
\begin{itemize}[wide]
\item $C_c(\Omega)$ --  the space of continuous functions with compact support in $\Omega$;
    \item $C_{\per}(Y)$ -- the subspace of $C(\RR^d)$ of $Y$-periodic functions;
    \item $C^{\infty}_{\per}(Y)$ -- the subspace of $C^{\infty}(\RR^d)$ of $Y$-periodic functions;
    \item $H^1_{\per}(Y)$ -- the closure of $C^{\infty}_{\per}(Y)$ in the $H^1$-norm;
    \item
    
    $\mathcal{D}(\Omega, X)$ with $X$ being a Banach space -- the space of infinitely differentiable functions from $\Omega$ to $X$, whose  support is a compact set of $\mathbb{R}^d$ contained in $\Omega$.

    \item $L^p(\Omega, X)$ with $X$ being a Banach space and $1 \le p \le \infty$ -- the space of measurable functions $w \colon x \in \Omega \mapsto w(x) \in X$ such that
    $
    \norm{w}_{L^p(\Omega, X)}
    \coloneqq \left(\int_{\Omega} \norm{w(x)}^p_{X} \di x\right)^\frac{1}{p} < \infty.
    $

    \item $L^p_{\per}\left(Y, C(\bar{\Omega})\right)$ -- the space of measurable functions $w \colon y \in Y \mapsto w(\cdot,y) \in C(\bar{\Omega})$, such that 
    $w$ is $Y$-periodic with respect to $y$ and
    $
    \int_{Y} \left(\sup_{x \in \bar{\Omega}} \abs{w(x,y)}\right)^p \di y 
    < \infty.
    $
  \item $\calM(\Omega)$  -- the space of finite Radon measures on
    $\Omega$, i.e., the dual space of $C_c(\Omega)$. This space is
    equipped with the total variation norm 
    \begin{align*}
      \norm{\mu}_{\calM} \coloneqq
      \abs{\mu} (\Omega)
      =
      \int_{\Omega} \abs{\mu} (\di x) = \sup \left\{ \int_{\Omega} \phi
      (x) \mu (\di x) \colon \phi \in C_c(\Omega),
      \norm{\phi}_{L^{\infty}(\Omega)} \le 1 \right\}.
    \end{align*}
    We then define  
    \begin{align*}
      \calM (\Omega,\RR^d)
      &\coloneqq \left\{ \mu = \mu_i
        \ee^i \colon \mu_i \in \calM(\Omega), \text{ for all } 1 \le i \le
        d\right\}, \text{ and }\\
      \calM \left(\Omega, X(Y, \RR^d)\right)
      &\coloneqq
        \left\{ \mu \in \calM \left( \Omega \times Y,\RR^d
        \right)\colon \mu (x,\cdot)\in X(Y,\RR^d) \text{ for }x \in \Omega \right\},
    \end{align*}
    where $X(Y,\RR^d)$ is a given space of functions from $Y \to \RR^d$, cf. \cite{ferreiraReiteratedHomogenizationBV2012,ferreiraCharacterizationMultiscaleLimit2012,amarTwoscaleConvergenceHomogenization1998}.
  \item $\mathrm{BV}(\Omega, \RR^d)$ -- the space of $d-$dimensional
    vector valued $L^1$- functions, whose Jacobians (in distributional
    sense) are $d\times d-$matrices of finite Radon measures on
    $\Omega$, i.e, $\uu = u_i \ee^i \in \mathrm{BV} (\Omega, \RR^d)$
    iff $u_i \in L^1(\Omega)$ and its distributional derivative
    $\frac{\partial u_i}{\partial x_j} \in \calM (\Omega)$ for all
    $1 \le i,j \le d$. We say a sequence of distributions
    $\left\{ T_n \right\}$ in $\calD'(\Omega)$ converges to $T$ in
    distributional sense if $\left\langle T_n, \varphi \right\rangle
    \cv \left\langle T_n,\varphi \right\rangle$ for any $\varphi \in \calD(\Omega)$.

\end{itemize}

Fix $p > 1$ and  $0 \le \alpha \le \min \left\{ 1, p-1\right\}$. 
Suppose
$\bfa\colon Y \times \RR^d \to \RR^d$ satisfies
\begin{enumerate}[label=(A{\arabic*}),ref=\textnormal{(A{\arabic*})}]
\item \label{cond:a-periodic} Measurability and $Y-$periodicity: for
  any $\xi \in \RR^d$, the function $\xi \mapsto \bfa(\cdot, \xi)$ is
  measurable and
\begin{equation*}
\bfa (z + m \ee^k, \xi) = \bfa (z,\xi),
\end{equation*}
for all $z \in \RR^d$, $m \in \ZZ$, and $k \in \{ 1,\ldots, d \}$.

\item\label{cond:a-bound} Boundedness: for all $z \in Y$, there exists
  $\Lambda_*>0$, such that
\begin{equation*}
\abs{\bfa (z,0)}  \le \Lambda_{*}.
\end{equation*}

\item\label{cond:a-continuity} Continuity: for a.e. $z \in Y$, there
  exists $\Lambda_o > 0$ such that the function 
  $\xi \mapsto \bfa \left( z, \xi \right)$ satisfies
\begin{align*}
  \abs{\bfa(z, \xi_1) - \bfa(z, \xi_2)} 
  \le \Lambda_o \left( 1 + \abs{\xi_1}^2 + \abs{\xi_2}^2 \right)^{\frac{p-1-\alpha}{2}} \abs{\xi_1-\xi_2}^{\alpha},
\end{align*}
for all $\xi_1, \xi_2 \in Y$, where $\alpha$ is described above.

\item\label{cond:a-monotonicity} Monotonicity: there exists $\lambda_o > 0$
  such that for a.e. $z \in Y$, we have
\begin{align*}
\left[ \bfa(z, \xi_1) - \bfa (z, \xi_2) \right] \cdot (\xi_1-\xi_2)
  \ge \lambda_o \left( 1 + \abs{\xi_1}^2 + \abs{\xi_2}^2
  \right)^{\frac{p-2}{2}} \abs{\xi_1-\xi_2}^{2},
\end{align*}
for all $\xi_1, \xi_2 \in \RR^d$.
\end{enumerate}



We also introduce a few more definitions that are necessary for the elasticity
equation. For $\uu \in H^1(\Omega,\RR^d)$, we define the symmetric
gradient, also known as the {\it linearized strain tensor},
\begin{align}
\label{eq:6}
\DD(\uu) \coloneqq \frac{\nabla \uu + \nabla \uu^{\top}}{2}.
\end{align}
Let $0 < \lambda_e < \Lambda_e < \infty$. We denote by
$\fraM(\lambda_e, \Lambda_e)$ the set of all fourth-order tensors
$\bfB = \left( \bfB_{ijkh} \right)_{1 \le i,j,k,h \le d}$ satisfying

\begin{enumerate}[label=(B{\arabic*}),ref=\textnormal{(B{\arabic*})}]
  \item\label{cond:b-bound} Boundedness and measurability: there exists $\Lambda_e> 0$ such that
\begin{equation*}
\norm{\bfB}_{L^{\infty}} \le \Lambda_e.
\end{equation*}
  \item\label{cond:b-elliptic} Ellipticity: there exists
    $\lambda_e > 0$ such that, for all $d\times d-$matrices $\bfc$ and for all $x \in
    \RR^d$, we have
\begin{equation*}
\bfB (x) \bfc : \bfc \ge \lambda_e \abs{\bfc}^2,
\end{equation*}
where ``$:$" represents the Frobenius inner product.
\end{enumerate}
Here, we recall that for two matrices $\bfc$ and $\bfd$,
$\bfB \bfc \coloneqq \left( (\bfB_{ijkh} \bfc_{kh})_{ij} \right)_{1
  \le i,j \le d}$ and
$\bfB \bfc : \bfd \coloneqq \bfB_{ijkh} \bfc_{ij} \bfd_{kh}$, with $1\leq i,j,k,h\leq d$.  We say
that a fourth-order tensor $\bfB$ is symmetric if
$\bfB_{ijkh} = \bfB_{jikh}=\bfB_{ijhk}$ for any $1 \le i,j,k,h \le
d$.
We denote by $\fraM_{\per}(\lambda_e,\Lambda_e)$ the subset of
$\fraM(\lambda_e,\Lambda_e)$ consisting of $Y-$periodic tensors, and by $ \fraM_{\sym}(\lambda_e,\Lambda_e)$ the subset of
$\fraM(\lambda_e,\Lambda_e)$ consisting of symmetric tensors.
We define
$\fraM_{\mat} \coloneqq \fraM_{\sym} \cap \fraM_{\per}$.

\subsubsection{The two-scale convergence method}
\label{sec:two-scale-conv}

The proof of the main result in this paper will be based on the theory of {\it two-scale
  convergence} that was first introduced  by G. Nguetseng \cite{nguetsengGeneralConvergenceResult1989} and further
developed by G. Allaire \ynote{\cite{allaireHomogenizationTwoscaleConvergence1992}}.  In this section, we present important
definitions and results which are relevant to this paper, and whose
proofs can be found in
\cite{allaireHomogenizationTwoscaleConvergence1992,nguetsengGeneralConvergenceResult1989,amarTwoscaleConvergenceHomogenization1998,ferreiraReiteratedHomogenizationBV2012,davoliHomogenizationBVModel2021
,lukkassenTwoscaleConvergence2002,visintinTwoscaleCalculus2006,cioranescuPeriodicUnfoldingHomogenization2002}.

\begin{definition}[$L^p$-admissible test function]
\label{def:admissible-function}
Let $1 \le p < + \infty$. A function $\psi \in
L^p(\Omega \times Y)$,  $Y$-periodic in the second variable, is called an $L^p$-admissible test function if, for
all $\varepsilon > 0$,
$\psi \left( \cdot, \frac{\cdot}{\varepsilon} \right)$ is measurable and
\begin{align}
\label{eq:34}
\lim_{\varepsilon \to 0} \int_{\Omega} \abs{\psi \left( x,
  \frac{x}{\varepsilon} \right)}^p \di x = \frac{1}{\abs{Y}}
  \int_{\Omega} \int_Y \abs{\psi (x,y)}^p \di y \di x.
\end{align}
\end{definition}

It is known that functions that belong to the spaces $\mathcal{D} \left(\Omega,
  C_\per^\infty (Y)\right)$, $C \left( \bar{\Omega}, C_{\per}(Y)
\right)$, $L^p_{\per}\left( Y,  C(\bar{\Omega})\right)$ or $L^p\left(\Omega, C_{\per}(Y)\right)$ are admissible \cite{allaireHomogenizationTwoscaleConvergence1992}, but the precise
characterization of admissible test functions is still an
open question.

\begin{definition}[Two-scale convergence]
  \label{def:two-scale}
A sequence $\{ v^\varepsilon \}_{\varepsilon>0}$ in $L^p(\Omega)$
 ($\calM (\Omega)$, respectively) is said to \emph{(weakly) two-scale converge}
 to $v = v(x,y)$ in $L^p (\Omega \times Y)$ ($\calM (\Omega)$, respectively), with $v \in L^p (\Omega \times Y)$ ($\calM (\Omega)$, respectively), and we write $v^\varepsilon \wcv[2] v$ in $L^p (\Omega \times Y)$, if and only if:
\begin{align}
\label{eq:2sc}
    \lim_{\varepsilon \to 0} \int_\Omega v^\varepsilon(x) \psi \left( x, \frac{x}{\varepsilon}\right) \di x 
    = \frac{1}{\abs{Y}} \int_\Omega \int_Y v(x,y) \psi(x,y) \di y \di x,
\end{align}
for any test function $\psi = \psi (x, y)$ with $\psi \in
\mathcal{D} \left(\Omega, C_\per^\infty (Y)\right)$.

In particular, if $v^{\varepsilon} \wcv[2] v$ and
$\norm{v^{\varepsilon}}_{L^p(\Omega)} \cv \norm{v}_{L^p(\Omega \times
  Y)}$, then we say that $v^{\varepsilon}$ \emph{strongly two-scale
  converges} to $v$ in $L^p (\Omega \times Y)$, and we write $v^{\varepsilon} \cv[2] v$ in $L^p (\Omega \times Y)$.
\end{definition}
We note that any bounded sequence $v^{\varepsilon}\in L^p(\Omega)$, with $1 < p < \infty$, ($\calM (\Omega)$, respectively) has a
subsequence that two-scale converges to a limit
$v^0 \in L^p(\Omega \times Y)$ ($\calM (\Omega \times Y)$,
respectively), cf. \cite{allaireHomogenizationTwoscaleConvergence1992,nguetsengGeneralConvergenceResult1989,amarTwoscaleConvergenceHomogenization1998,ferreiraReiteratedHomogenizationBV2012,davoliHomogenizationBVModel2021,ferreiraCharacterizationMultiscaleLimit2012,visintinTwoscaleCalculus2006,cioranescuPeriodicUnfoldingHomogenization2002,lukkassenTwoscaleConvergence2002}.

The strong two-scale convergence plays an important role in
establishing corrector results in homogenization, cf. \cite[Theorem
1.8]{allaireHomogenizationTwoscaleConvergence1992} and \cite[Theorem 11]{lukkassenTwoscaleConvergence2002}:

\begin{lemma}
\label{sec:two-scale-corrector}
Let $\left\{ v^{\varepsilon} \right\}_{\varepsilon > 0}$ be a sequence
in $L^p(\Omega)$ that strongly two-scale converges to $v \in
L^p(\Omega \times Y)$. Suppose further that $v$ is admissible, in the
sense of \cref{def:admissible-function}. Then 
\begin{align}
\label{eq:19}
\lim_{\varepsilon \to 0} \norm{v^{\varepsilon}(x) - v \left( x,
  \frac{x}{\varepsilon} \right)}_{L^p(\Omega)} = 0.
\end{align}
\end{lemma}

For $v \in L^p(\Omega \times Y)$, the function $v \left( x,\frac{x }{\varepsilon} \right)$ is not necessarily a measurable function
\cite{allaireHomogenizationTwoscaleConvergence1992,visintinTwoscaleCalculus2006}. This explains why it is required for $v$ to be admissible in
\cref{sec:two-scale-corrector}. We can circumvent this assumption by
introducing the so-called \emph{coarse-scale averaging operator}
$\MM^{\varepsilon}$, which will be defined next. For each
$z \in \RR^d$, let $[z]_Y$ be the integer part of $z$, more precisely,
$[z]_Y \coloneqq k_i \ee^i$ with $k_i \in \ZZ$ such that
$\left\{ z \right\}_Y \coloneqq z - [z]_Y \in Y$. For each function
$v \in L^p(\Omega \times Y)$, let
\begin{align}
\label{eq:22}
  (\MM^{\varepsilon} v) (x,y)
  \coloneqq \int_Y v \left( \varepsilon\left[ \frac{x}{\varepsilon}
  \right]_Y + \varepsilon \xi,y \right) \di \xi, \quad \text{for } (x,y)
  \in \Omega \times Y.
\end{align}
It is known that $\MM^{\varepsilon}$ is a bounded linear operator and that
$(\MM^{\varepsilon} v) (x,y)$ and $(\MM^{\varepsilon} v)
\left( x, \frac{x}{\varepsilon} \right)$ are both measurable
\cite{visintinTwoscaleCalculus2006}. The following result is due to Visintin \cite[Proposition
2.3]{visintinTwoscaleCalculus2006}:
\begin{lemma}
\label{sec:two-scale-conv-5}
A sequence $\left\{ v^{\varepsilon} \right\}_{\varepsilon > 0} \subset
L^p(\Omega)$ strongly two-scale converges to $v \in L^p(\Omega \times
Y)$ if and only if 
\begin{align*}
\lim_{\varepsilon \to 0} \norm{v^{\varepsilon}(x) - \left(
  \MM^{\varepsilon} v \right) \left( x,
  \frac{x}{\varepsilon} \right)}_{L^p(\Omega)} = 0.
\end{align*}
In particular, if $v$ is admissible, then the operator
$\MM^{\varepsilon}$ can be dropped.
\end{lemma}


\subsection{Formulation of the problem: the fine-scale coupled system}
\label{sec:fine-scale-coupled}
In this section, we set up the fine-scale problem.  Denote by 
$ \bfC\in\fraM_{\per}(\lambda_e,\Lambda_e)$ the {\it electrostriction} tensor and by $\bfB \in\fraM_{\mat}(\lambda_e,\Lambda_e)$ the {\it elasticity} tensor. 
We assume further
that $Y$ is a disjoint union of finite subdomains with piecewise $C^{1,\alpha}-$boundaries and $\bfB$ is
H\"{o}lder continuous on the closure of each subdomain.
Let
$\bfa\colon Y \times \RR^d \to \RR^d$ satisfying
\ref{cond:a-periodic}--\ref{cond:a-monotonicity}, which pertains to a {(nonlinear)} constitutive law between the electric displacement
and the electric field. Furthermore, let $\bg \in L^r(\Omega,\RR^d)$
for some $r > 1$, and $f \in L^{p'}(\Omega)$, where $p'$ is the
H\"{o}lder conjugate of $p$, i.e. $\frac{1}{p}+\frac{1}{p'}=1$. Then
the fine-scale displacement $\uu^{\varepsilon} \in \mathrm{BV} (\Omega,\RR^d)$
and the electrostatic potential
$\varphi^{\varepsilon} \in W_0^{1,p}(\Omega)$ satisfy the following
coupled system
\cite{tianDielectricElastomerComposites2012,francfortEnhancementElastodielectricsHomogenization2021,toupinElasticDielectric1956}:
\begin{subequations}
  \label{eq:p410} 
  \begin{align}
    \label{eq:p421}
  -\Div \left[ \bfa \left( \frac{x}{\varepsilon}, \nabla
  \varphi^{\varepsilon}  \right)\right]
  &=  f && \text{ in }\Omega,\\
\label{eq:p411}
  -\Div \left[ \bfB \left( \frac{x}{\varepsilon} \right) \nabla
  \uu^{\varepsilon} + 
  \bfC \left( \frac{x}{\varepsilon} \right) \left( \nabla \varphi^{\varepsilon} \otimes \nabla \varphi^{\varepsilon} \right) \right]
  &= \bg && \text{ in } \Omega,
\end{align}
\end{subequations}
together with the balance equations:
\begin{subequations}
 \label{eq:p413}
\begin{align}
  \label{eq:p414}
  \int_{\Gamma_{i}^{\varepsilon}} \left\llbracket
  \bfB \left( \frac{x}{\varepsilon} \right) \DD (\uu^{\varepsilon}) + 
  \bfC \left( \frac{x}{\varepsilon} \right) \left( \nabla \varphi^{\varepsilon} \otimes \nabla \varphi^{\varepsilon} \right)\right\rrbracket \nn_i\di \hmeas
  &=0,\\
  \label{eq:p422}
  \int_{\Gamma_{i}^{\varepsilon}} \left\{ \left\llbracket
  \bfB \left( \frac{x}{\varepsilon} \right) \DD
  (\uu^{\varepsilon}) + 
   \bfC \left( \frac{x}{\varepsilon} \right) \left( \nabla \varphi^{\varepsilon} \otimes \nabla \varphi^{\varepsilon} \right)\right\rrbracket\nn_i \right\} \times \nn_i\di
  \hmeas
  &=0,
\end{align}
\end{subequations}
and the boundary conditions:
\begin{subequations}
  \label{eq:p415} 
  \begin{align}
     \label{eq:p423}
  \varphi^{\varepsilon}
  &= 0, \text{ on }\partial \Omega,\\
\label{eq:p416}
  \uu^{\varepsilon}
  &= 0, \text{ on }\partial \Omega.
\end{align}
\end{subequations}
Here, $\left\llbracket \, \cdot \, \right\rrbracket$ denotes the jump
on the interfaces $\Gamma^{\varepsilon}_i$.

Next, we introduce the variational formulation for \eqref{eq:p410}-\eqref{eq:p415}. To
simplify the notation, we define 
\begin{align*}
  \maxss^{\varepsilon}
  \coloneqq \nabla \varphi^{\varepsilon} \otimes \nabla \varphi^{\varepsilon},
\end{align*}
which will be called the {\it Maxwell stress tensor} (it is worth
mentioning that the most general version of the Maxwell stress tensor has three
additional terms \cite{zangwillModernElectrodynamics2013}).  Then, we seek for
$\varphi^{\varepsilon} \in W_0^{1,p}(\Omega)$ and
$\uu^{\varepsilon} \in \mathrm{BV}(\Omega,\RR^d)$ such that, for all
$\eta \in W^{-1,p'}(\Omega)$ and
$\vv \in \calD(\Omega, \RR^d)$, the following holds
\begin{align}
  \label{eq:r5}
  \begin{split}
    \int_{\Omega} \bfa \left( \frac{x}{\varepsilon}, \nabla
  \varphi^{\varepsilon}  \right) \ynote{\cdot} \nabla \eta \di x
    + \int_{\Omega} \bfB \left( \frac{x}{\varepsilon} \right) \DD(\uu^{\varepsilon}) : \DD(\vv)
    \di x 
    +
    \int_{\Omega} \bfC \left( \frac{x}{\varepsilon} \right) \, \maxss^{\varepsilon}  :   \DD(\vv)\di x\\
    =  \int_{\Omega} \bg \cdot \vv \di x 
    + \int_{\Omega} f \, \eta \di x.
\end{split}
\end{align}

By setting $\vv = 0$, we seek for a unique solution
$\varphi^{\varepsilon}  \in W_0^{1,p}(\Omega)$ of the electrostatic problem
\begin{align}
  \label{eq:r6}
  \int_{\Omega} \bfa \left( \frac{x}{\varepsilon}, \nabla
  \varphi^{\varepsilon}  \right) \cdot \nabla \eta \di x
  =\int_{\Omega} f \, \eta \di x, \quad
  \forall \eta \in W^{-1,p'}(\Omega).
\end{align}

By setting $\eta = 0$, we seek for $\uu^{\varepsilon} \in \mathrm{BV}(\Omega,\RR^d)$ such
that
\begin{align}
\label{eq:r7}
\int_{\Omega} \bfB \left( \frac{x}{\varepsilon} \right) \, \DD(\uu^{\varepsilon}) : \DD(\vv)
  \di x 
  =\int_{\Omega} \bg \cdot \vv \di x
  - \int_{\Omega} \bfC \left( \frac{x}{\varepsilon} \right) \,
  \maxss^{\varepsilon} : \DD(\vv)\di x, \quad \forall \, \vv \in \calD(\Omega, \RR^d).
\end{align}

\section{Main results}
\label{sec:main-result}

This section is dedicated to presenting the main outcomes of this paper. Before we state our main theorem, which provides the homogenization result for the coupled system \eqref{eq:p410}, we will introduce a \emph{first-order corrector} for the solution $\varphi^{\varepsilon} \in
W_0^{1,p}(\Omega)$ of the electrostatic equation.

To begin, we consider the homogenization of the electrostatic equation
using its variational formulation \eqref{eq:r6}. The proof can
  be found in \cref{sec:proof-crefs-results-2}.
\begin{proposition}
  \label{sec:main-results-7}
  Let $\varphi^{\varepsilon} \in W_0^{1,p}(\Omega)$ be the unique solution of
  \eqref{eq:r6}. There exist $\varphi^0 \in W_0^{1,p}(\Omega)$ and $\varphi^1 \in
L^p(\Omega, W_{\per}^{1,p}(Y)/\RR)$ such that 
\begin{align}
\label{eq:4}
\varphi^{\varepsilon} \wcv[2] \varphi^0, \text{ and } \nabla
  \varphi^{\varepsilon} \wcv[2] \nabla \varphi^0 + \nabla_y \varphi^1,
\end{align}
where the two-scale limits $\varphi^0$
and $\varphi^1$ satisfy the
following system
\begin{subequations}
  \label{eq:31}
  \begin{align}
    \label{eq:26}
    -\Div \left[ \frac{1}{\abs{Y}}\int_Y \bfa\left(y, \nabla \varphi^0 +
    \nabla_{\yy} \varphi^1 \right)\di \yy \right]
    &= f, \text{ in } \Omega,\\
    \label{eq:28}
    -\Div_{\yy}  \bfa \left(y, \nabla \varphi^0 + \nabla_{\yy}
    \varphi^1\right) 
    &=0, \text{ in } \Omega \times Y.
  \end{align}
\end{subequations}
\end{proposition}

We now state the first main result of this paper - an explicit \emph{corrector result}
for the electrostatic problem, whose proof can be found in
\cref{sec:proof-crefs-results}.

\begin{theorem}[Explicit corrector result]
\label{sec:main-results-1}
Suppose $p > 1.$ Let $(\varphi^0,\varphi^1) \in  W_0^{1,p}(\Omega) \times L^p(\Omega,
W^{1,p}_{\per}(Y)/\RR)$ be the solution of \eqref{eq:31}. Then for $1\leq i\leq d$,
\begin{align}
\label{eq:112}
  \frac{\partial \varphi^{\varepsilon}}{\partial x_i}
  \cv[2] \frac{\partial \varphi^0}{\partial x_i} + \frac{\partial
  \varphi^1}{\partial y_i} \qquad \text{ in } L^p(\Omega \times Y).
\end{align}
As a consequence,
\begin{align}
\label{eq:44}
\lim_{\varepsilon \to 0}\norm{\nabla \varphi^{\varepsilon}(x)  -\nabla \varphi^0(x) -
  \MM^{\varepsilon} (\nabla_y \varphi^1)\left( x,\frac{x}{\varepsilon} \right)}_{L^{p}(\Omega,\RR^d)}
  =0.
\end{align}
In particular, if $\varphi^1$ is admissible, then 
\begin{align}
\label{eq:23}
\lim_{\varepsilon \to 0}\norm{\nabla \varphi^{\varepsilon}(x)  -\nabla \varphi^0(x) -
  \nabla_y \varphi^1 \left( x,\frac{x}{\varepsilon} \right)}_{L^{p}(\Omega,\RR^d)}
  =0.
\end{align}


\end{theorem}

\begin{remark}
  \label{sec:main-results-6}
  Theorem \ref{sec:main-results-1} implies the corrector result for the linear case, i.e.,
  when $\bfa(y, \xi) = \bfb (y) \xi$ for some bounded, elliptic, and
  $Y$-periodic matrix $\bfb$. Indeed,  fine-scale equation 
\begin{align}
\label{eq:114}
-\Div \left( \bfb \left( \frac{x}{\varepsilon} \right) \nabla
  \varphi^{\varepsilon} \right) = f \text{ in }\Omega, \qquad
  \varphi^{\varepsilon} = 0 \text{ on }\partial \Omega,
\end{align}
has a unique solution $\varphi^{\varepsilon}\in H^1_0(\Omega)$ such that
\begin{align*}
  \varphi^{\varepsilon}
    &\wcv \varphi^0 \quad \text{ in } H_0^1(\Omega),\\
  \nabla \varphi^{\varepsilon}
    &\wcv \nabla \varphi^0 + \nabla_y \varphi^1 \quad \text{ in } L^2 \left(
    \Omega \times Y, \RR^d \right),
\end{align*}
for some $\varphi^0 \in H_0^1(\Omega)$ and $\varphi^1 \in L^2 \left(
  \Omega, H_{\per}^1(Y)/\RR \right)$ that satisfy
\begin{align*}
-\Div \left[ \frac{1}{\abs{Y}} \bfb (y) \left( \nabla \varphi^0 +
  \nabla_y \varphi^1 \right) \right]
  &= f, \qquad \text{ in }\Omega,\\
  -\Div_y \left[ \bfb(y) \left( \nabla \varphi^0 + \nabla_y \varphi^1
  \right) \right]
  &=0, \qquad \text{ in }\Omega \times Y.
\end{align*}
By introducing the cell problems 
\begin{align*}
  \omega^i \in H_{\per}^1(Y)/\RR,\qquad
  -\Div_y \left[ \bfb(y) (\ee^i +  \nabla_y \omega^i(y)) \right] = 0
  \quad \text{ in }Y, \quad 1\leq i\leq d,
\end{align*}
 we have $\varphi^0\in H^1_0(\Omega)$ and $\varphi^1\in L^2 \left(
  \Omega, H_{\per}^1(Y)/\RR \right)$ satisfy
\begin{align}
\label{eq:115}
  \begin{split}
  -\Div \left( \bfb^{\hom} \nabla \varphi^0  \right)
  &=f \text{ in }\Omega,\\
  \varphi^1(x,y)
  &= \frac{\partial \varphi^0}{\partial x_i}(x) \, \omega^i(y),
  \end{split}
\end{align}
where 
\begin{align*}
  \bfb^{\hom}_{jk}
  = \frac{1}{\abs{Y}} \int_Y \bfb(y) \left( \ee^k + \nabla \omega^k(y)
  \right) \cdot \left( \ee^j + \nabla \omega^j(y) \right) \di y, \quad
  1\leq j,k\leq d.
\end{align*}
Since $\bfb^{\hom}$ is a constant elliptic matrix, if $f$ is smooth
enough (e.g., $f \in L^{\infty}(\Omega)$), then $\varphi^0 \in C
(\bar{\Omega})$, and therefore, $\varphi^1 \in L^2(Y, C
(\bar{\Omega}))$, or $\varphi^1$ is admissible. Therefore,
\eqref{eq:23} holds, and we recover the classical corrector result for
the linear case.

In general, when $\bfa$ is nonlinear, $\varphi^1$ does not admit a
finite representation as in \eqref{eq:115}, and thus, the
admisssibility assumption is necessary.

\end{remark}




\begin{remark}
  \label{sec:electr-equat}
 In contrast to our \cref{sec:main-results-1}, in the existing literature, the corrector results for monotone operators are not explicit, in the sense that they are written in the following form:
\begin{align*}
\lim_{\varepsilon \to 0}\norm{\nabla \varphi^{\varepsilon} - \bfp\left(\frac{\cdot}{\varepsilon},
  M^{\varepsilon}(\nabla \varphi^0)\right)}_{L^p(\Omega,\RR^d)}
  = 0,
\end{align*}
where the functions $\bfp$ and $M^{\varepsilon}$ will be
defined in \cref{sec:prel-estim}, cf. e.g.,
\cite{dalmasoCorrectorsHomogenizationMonotone1990} (see also
\cite[Remark 3.7]{allaireHomogenizationTwoscaleConvergence1992}) and
references therein.
\end{remark}

\begin{remark}
  In \cite[Theorem 3.6]{allaireHomogenizationTwoscaleConvergence1992},
  the author obtained \eqref{eq:23} for $p =2$ under the assumption
  that $\varphi^1(x,y)$ is admissible in the sense of
  \cref{def:admissible-function}.  In \cref{sec:main-results-1}, we
  are able to prove that \eqref{eq:23} holds for any
  $p >1$. The corrector result obtained in
  \cref{sec:main-results-1} is fundamental to obtain the closed system
  in \cref{thm:main} below. It is also useful in 
  numerical computation 
  and large-scale regularity of
  homogenization of monotone problems
  \cite{efendievMeyersTypeEstimates2006,bunderEquationfreePatchScheme2021,bunderNonlinearEmergentMacroscale2021,clozeauQuantitativeNonlinearHomogenization2023}.

  Finally, the arguments used in 
  \cref{sec:main-results-1} can be employed to improve several
  classical results, which will be presented in
  \cref{sec:some-inter-corr}.
\end{remark}

We now introduce some auxiliary problems and definitions that will be
necessary for the statement of the main theorem.

To write the homogenized electrostatic equation, in \eqref{eq:28}, we
replace $\nabla \varphi^0$ by $\xi \in \RR^d$ and let
$\varphi^1 = \xi \eta_{\xi}$ for some
$\eta_{\xi} \in W^{1,p}_{\per}(Y)/\RR$ to obtain the cell problem
\begin{align}
\label{eq:13}
-\Div_y  \bfa \left( y, \xi + \nabla_y \eta_{\xi} \right) = 0.
\end{align}

For $1 \le i,j \le d,$ denote by
$\bfU^{ij}$ the vector defined by
$\bfU^{ij}_k \coloneqq y_j \delta_{ik}$ and consider
$\vec{\Upsilon}^{ij} \in W^{1,p}_{\per}(Y,\RR^d)/\RR$ solving
\begin{align}
\label{eq:504}
  \begin{split}
    \Div_{y} \left[ \bfB (y) \, \DD_{y} \left(\bfU^{ij} - \vec{\Upsilon}^{ij}
      \right)\right]
    &= 0 \text{ in } Y,\\
    \int_{\Gamma}  \left\llbracket
      \bfB(y) \, \DD_{y} \left( \bfU^{ij} - \vec{\Upsilon}^{ij}
      \right)
    \right\rrbracket \nn_{\Gamma} \di \hmeas 
    &=0,\\
    \int_{\Gamma}  \left\llbracket
      \bfB(y) \DD_{y} \left( \bfU^{ij} - \vec{\Upsilon}^{ij}
      \right)
    \right\rrbracket \nn_{\Gamma} \times \nn_{\Gamma} \di \hmeas 
    &=0,
    \end{split}
    \end{align}
and, also, consider $\vec{\chi}^{ij} \in W^{1,p}_{\per}(Y,\RR^d)/\RR$
solving
\begin{align}
\label{eq:504b}
  \begin{split}
    \Div_{y} \left[ \bfC(y) \, \DD_{y} \left( \vec{\chi}^{ij}
      \right) + \vec{\zeta}^{ij} \right]
    &= 0 \text{ in } Y,\\
    \int_{\Gamma}  \left\llbracket
      \bfC(y)\DD_{y} \left( \vec{\chi}^{ij}
      \right) +  \vec{\zeta}^{ij}
    \right\rrbracket \nn_{\Gamma} \di \hmeas  
    &=0,\\
    \int_{\Gamma}  \left\llbracket
      \bfC(y) \DD_{y} \left( \vec{\chi}^{ij}
      \right) +  \vec{\zeta}^{ij}
    \right\rrbracket \nn_{\Gamma} \times \nn_{\Gamma} \di \hmeas 
    &=0,
  \end{split}
\end{align}
with the matrix $\vec{\zeta}^{ij}$ given by 
\begin{equation}
    \label{zetaij}
    \vec{\zeta}^{ij}(y)
  \coloneqq
   (\ee^i + \nabla_{y} \eta_{\ee^i}(y))  \otimes ( \ee^j
    + \nabla_{y} \eta_{\ee^j}(y)),~ y \in Y,
\end{equation}
where $\eta_{\ee^j}$ given by \eqref{eq:13} denotes the microscopic electric stress tensor on $Y$. 

\begin{remark}
\label{sec:main-results}
The cell problem \eqref{eq:504} describes the local behavior of the elastic displacement without the electric effect. The electric effect is
captured in \eqref{eq:504b}. Note that the above cell problems (or
similar) were observed in
\cite{francfortEnhancementElastodielectricsHomogenization2021,dangHomogenizationNondiluteSuspension2021,dangGlobalGradientEstimate2022}.
\end{remark}

We also define: 
\begin{align}
\label{eq:513}
  \begin{split}
    \bfa^{\hom} (\xi)
  &\coloneqq \frac{1}{\abs{Y}} \int_Y \bfa \left( y, \xi + \nabla_y
    \eta_{\xi}(y) \right) \di y, \quad \xi \in \RR^d,\\
  \bfB^{\hom}_{ijmn}
  &\coloneqq \frac{1}{\abs{Y}} \int_Y 
  \bfB(y) \, \DD_{y}(\bfU^{ij}-\vec{\Upsilon}^{ij}) : \DD_{y}(\bfU^{mn}-\vec{\Upsilon}^{mn})
  \di y, \quad 1 \le i,j,m,n \le d,\\
    \bfC^{\hom}_{ij}
  &\coloneqq
    \frac{1}{\abs{Y}} \int_Y \left( \bfC(y) \, \DD_{y}(\vec{\chi}^{ij}) 
    +  \vec{\zeta}^{ij}\right)\di y, \quad 1 \le i,j \le d,
    \end{split}
\end{align}
where $\bfa^{\hom}$ is the \emph{effective electric conductivity},
which is monotone. The tensor $\bfB^{\hom} \coloneqq \left\{ \bfB^{\hom}_{ijmn}\right\}_{1 \le i,j,m,n \le
  d}$ is the \emph{effective elasticity}, and it is a fourth-rank
tensor that is symmetric and elliptic.
And lastly, $\bfC^{\hom}$ is the \emph{effective electrostriction tensor}.

We now state the main result of this paper, \cref{thm:main}, which provides
the homogenization for the coupled system \eqref{eq:p410}. The
proof of \cref{thm:main} is carried out in \cref{sec:proof-crefthm:main}.

\begin{theorem} \label{thm:main}
Suppose further that $p \ge 2$. The solution
  $(\varphi^{\varepsilon}, \uu^{\varepsilon}) \in W_0^{1,p}(\Omega)
  \times W_0^{1,1}(\Omega,\RR^d)$ of \eqref{eq:p410} satisfies
\begin{align*}
    \varphi^\varepsilon 
    \wcv \varphi^0 \text{ in } W_0^{1,p}(\Omega), \qquad \text{
  and } \qquad
    \uu^\varepsilon
    \wcv \uu^0 \text{ in distribution},
\end{align*}
where  $\uu^0\in W_0^{1,1}(\Omega,\RR^d)$ and $\varphi^0\in W_0^{1,p}(\Omega)$ are solutions of
\begin{align}
\label{summary-eqn}
\begin{aligned}
  -\Div \left[ \bfa^{\hom}( \nabla\varphi^0) \right] 
  &= f &&\text{ in }\Omega,\\
  -\Div \left[ \bfB^{\hom}\,\DD \left( \uu^0 \right)+\bfC^{\hom} (
  \nabla \varphi^0 \otimes \nabla \varphi^0 )\right]
  &= \bg &&\text{ in }\Omega,\\
  \uu^0
  = 0, \quad \varphi^0
  &= 0,  && \text{ on } \partial \Omega,
  \end{aligned}
\end{align}
with $\bfa^{\hom},\bfB^{\hom},$ and $\bfC^{\hom}$ defined in \eqref{eq:513}.

Moreover, 
\begin{itemize}
\item  $\bfa^{\hom}$
is continuous and monotone, i.e., for every $\xi_1, \xi_2 \in \RR^d$, the following holds
\begin{align}
\label{eq:48}
  \begin{split}
    \abs{\bfa^{\hom}(\xi_1) - \bfa^{\hom}(\xi_2)}
    &\le \Lambda_{\hom}\left( 1 + \abs{\xi_1}^2 + \abs{\xi_2}^2 \right)^{\frac{p-2-\theta}{2}} \abs{\xi_1-\xi_2}^{\theta},\\
    \left[ \bfa^{\hom}(\xi_1) - \bfa^{\hom}(\xi_2) \right]\cdot \left(
    \xi_1 -\xi_2 \right)
    &\ge \lambda_{\hom} \left( 1 + \abs{\xi_1}^2 + \abs{\xi_2}^2 \right)^{\frac{p-2}{2}} \abs{\xi_1-\xi_2}^{2},
  \end{split}
\end{align}
where $\theta = \frac{\alpha}{2-\alpha},$ and $\lambda_{\hom},
\Lambda_{\hom} > 0$ are constants depending on $d, p, \alpha, \beta,
\lambda_o, \Lambda_o, \Lambda_{*}$, which are defined in \cref{ss:setup}.
\item  $\bfB^{\hom}$ is symmetric and elliptic.
\end{itemize}

\end{theorem}

\begin{remark}
\label{sec:main-results-3}

The result in \cref{thm:main} holds also for non-homogeneous Dirichlet
boundary conditions. In fact, \cref{thm:main} holds if
$\varphi^{\varepsilon} = \varphi_b$ on $\partial \Omega$, for some
$\varphi_b \in W^{1,p^+}(\Omega), p^+ > p$, and \ref{cond:a-bound} is
  replaced by
\begin{align}
\label{eq:98}
\abs{\bfa(y, \nabla\varphi_b(x))} \le \Lambda_{*}, \qquad \text{a.e. }
  y \in Y, x\in \Omega;
\end{align}
with all the estimates changed accordingly to adapt to this
case. For example, the boundedness of $\varphi^{\varepsilon}$ can be
proven as follows. By the monotonicity \ref{cond:a-monotonicity} and the boundedness \ref{cond:a-bound} of $\bfa$, and using H\"{o}lder, Young and Poincar\'{e}
inequalities, we obtain
\begin{align*}
  \lambda_o \norm{\nabla \varphi^{\varepsilon}-\nabla \varphi_{b}}^p_{L^p}
  &\le \int_{\Omega} \left[  \bfa \left( \frac{x}{\varepsilon}, \nabla
    \varphi^{\varepsilon}  \right) - \bfa \left(
    \frac{x}{\varepsilon},\nabla\varphi_b \right) \right] \cdot (\nabla
    \varphi^{\varepsilon}-\nabla \varphi_b) \di x\\
 &=\int_{\Omega} f \left( \varphi^{\varepsilon} - \varphi_b \right)\di x-\int_{\Omega} \bfa \left(
    \frac{x}{\varepsilon},\nabla\varphi_b \right) \cdot \left( \nabla
    \varphi^{\varepsilon} - \nabla \varphi_b \right) \di x\\
  &\le
    \norm{f}_{L^{p'}}\norm{\varphi^{\varepsilon}-\varphi_b}_{L^p} + C \Lambda_{*}
    \norm{\nabla \varphi^{\varepsilon}-\nabla \varphi_b}_{L^p}\\
  &\le C (\norm{f}_{L^{p'}} + \Lambda_{*}) 
    \norm{\nabla \varphi^{\varepsilon}-\nabla \varphi_b}_{L^p}\\
  &\le \frac{C}{p'\mu} \left( \norm{f}_{L^{p'}} + \Lambda_{*}
    \right)^{p'} + C \, \frac{\mu}{p} \norm{\nabla\varphi^{\varepsilon} - \nabla \varphi_b}_{L^p}^p,
\end{align*}
for some $\mu > 0$. Choosing $\mu$ small enough, we conclude 
\begin{align*}
  \norm{\nabla \varphi^{\varepsilon} - \nabla \varphi_b}^p_{L^p}
  \le
  C \left( \norm{f}_{L^{p'}} + \Lambda_{*} \right)^{p'}
  \le C \left( \norm{f}_{L^{p'}}^{p'} + \Lambda_{*}^{p'} \right);
\end{align*}
and thus by Poincar\'{e} inequality, 
\begin{align*}
\norm{\varphi^{\varepsilon}}_{W^{1,p}}^p
  &\le C\norm{\varphi^{\varepsilon}- \varphi_b}_{W^{1,p}}^p
    + C\norm{\varphi_b}_{W^{1,p}}^p \\
  &\le C\norm{\nabla \varphi^{\varepsilon} - \nabla \varphi_b}^p_{L^p} +
  C\norm{\varphi_b}_{W^{1,p}}^p\\
  &\le
  C \left( \norm{f}^{p'}_{L^{p'}} + \norm{ \varphi_b}^p_{W^{1,p}}+ \Lambda_{*}^{p'} \right).
\end{align*}
\end{remark}

\begin{remark}
\label{sec:main-results-5}
The homogenization results of the coupled systems in \cite{francfortEnhancementElastodielectricsHomogenization2021,tianDielectricElastomerComposites2012} follow directly as particular cases of \cref{thm:main}, when $p=2$ and $\bfa$ is linear, i.e.,
$\bfa (y, \xi) = \bfb (y) \xi$ for some bounded, elliptic, and
periodic matrix $\bfb$.
\end{remark}

\begin{remark}
The enhancement effect, i.e., when the right hand side of \eqref{eq:p421} has the form $f(x) = \frac{1}{\varepsilon} f_1 (x) f_2 (\frac{x}{\varepsilon})$ for some $f_1$ and $f_2$, as considered in
\cite{francfortEnhancementElastodielectricsHomogenization2021}, can be
handled using the approach developed in this paper.
\end{remark}
\begin{remark}
The proof in this paper can be  straightforwardly extended to the case
when $\bfB$ is a function in $VMO (\Omega,\RR^d)$,
see \cref{sec:high-regul-exist}.
\end{remark}

\section{Proof of \cref{sec:main-results-7}}
\label{sec:proof-crefs-results-2}

In \eqref{eq:r6}, let
$\eta = \varphi^{\varepsilon} \in W^{1,p}_0(\Omega)$ and use \ref{cond:a-bound}, \ref{cond:a-monotonicity}, H\"{o}lder inequality, and
Poincar\'{e} inequality, to obtain
\begin{align*}
  \lambda_o \norm{\nabla \varphi^{\varepsilon}}^p_{L^p(\Omega,\RR^d)}
  &\le \int_{\Omega} \left[  \bfa \left( \frac{x}{\varepsilon}, \nabla
    \varphi^{\varepsilon}  \right) - \bfa \left(
    \frac{x}{\varepsilon},0 \right) \right] \cdot \nabla
    \varphi^{\varepsilon} \di x\\
  &=\int_{\Omega} f \, \varphi^{\varepsilon}\di x - \int_{\Omega}\bfa \left(
    \frac{x}{\varepsilon},0 \right) \cdot \nabla
    \varphi^{\varepsilon} \di x \\
    &\le
    \norm{f}_{L^{p'}(\Omega)}\norm{\varphi^{\varepsilon}}_{L^p(\Omega)} + C
      \norm{\nabla \varphi^{\varepsilon}}_{L^p(\Omega,\RR^d)}
    \le C \norm{f}_{L^{p'}(\Omega)}\norm{\nabla \varphi^{\varepsilon}}_{L^p(\Omega,\RR^d)}.
\end{align*}
The estimate above and Poincar\'{e} inequality deliver 
\begin{align}
\label{eq:32}
\norm{\varphi^{\varepsilon}}^p_{W_0^{1,p}(\Omega)} \le C \norm{f}^{p'}_{L^{p'}(\Omega)}.
\end{align}
Therefore, there exist $\varphi^0 \in W_0^{1,p}(\Omega)$ and $\varphi^1 \in
L^p(\Omega, W_{\per}^{1,p}(Y)/\RR)$ such that 
\begin{align*}
\varphi^{\varepsilon} \wcv[2] \varphi^0, \text{ and } \nabla
  \varphi^{\varepsilon} \wcv[2] \nabla \varphi^0 + \nabla_y \varphi^1.
\end{align*}

Since $F^{\varepsilon}(x) \coloneqq \bfa \left( \frac{x}{\varepsilon},
  \nabla \varphi^{\varepsilon}(x) \right)$ is
bounded in $L^{p'}(\Omega,\RR^d)$ by \eqref{eq:32} and
\ref{cond:a-continuity}, then there exists $F^0 \in L^{p'}(\Omega \times
Y)$ such that  
\begin{align}
\label{eq:125}
F^{\varepsilon} \wcv[2] F^0.
\end{align}
In \eqref{eq:r6}, by letting $\varepsilon \to 0$ and using \eqref{eq:125}, we
have
\begin{align*}
\frac{1}{\abs{Y}} \int_{\Omega \times Y} F^0(x,y) \cdot \nabla \eta (x) \di x
  \di y
  = \int_{\Omega} f(x) \eta (x) \di x, \qquad \text{ for all } \eta \in \calD (\Omega),
\end{align*}
and by applying Fubini's theorem and integration by parts, we obtain
\begin{align}
\label{eq:127}
  -\Div \left[ \frac{1}{\abs{Y}} \int_Y F^0(x,y) \di y \right]
  = f(x).
\end{align}

Let $\eta^0 \in \calD(\Omega)$ and $\eta^1 \in \calD (Y)$. In
\eqref{eq:r6}, choose
$\eta (x) = \varepsilon \eta^0(x)\eta^1 \left( \frac{x}{\varepsilon} \right)$,
then 
\begin{align*}
\int_{\Omega} \eta^0(x)F^{\varepsilon} \left( x \right) \cdot \nabla
  \eta^1 \left( \frac{x}{\varepsilon} \right) \di x + \varepsilon \int_{\Omega}
  \eta^1 \left( \frac{x}{\varepsilon} \right)F^{\varepsilon} \left( x \right) \cdot \nabla
  \eta^0 \left( x \right) \di x
  = \int_{\Omega} f (x) \varepsilon \eta^0(x) \eta^1 \left(
  \frac{x}{\varepsilon}\right) \di x.
\end{align*}
Letting $\varepsilon \to 0$ and using \eqref{eq:125}, we obtain
\begin{align*}
  \frac{1}{\abs{Y}}\int_{\Omega \times Y} \eta^0(x) F^0(x,y) \cdot \nabla
  \eta^1(y) \di x\di y = 0.
\end{align*}
By Fubini's theorem and the fundamental lemma of calculus of variation, we have
\begin{align*}
  \int_Y F^0(x,y) \cdot \nabla \eta^1(y) \di y
  = 0.
\end{align*}
Integration by parts delivers
\begin{align}
  \label{eq:126}
  \Div_y F^0(x,y)
   = 0.
\end{align}

We claim that
$F^0(x,y) = \bfa \left( y, \nabla \varphi^0(x) + \nabla_y
  \varphi^1(x,y) \right).$ Indeed, for $t > 0$, let
$\eta,\eta^1 \in \calD \left( \Omega, C_{\per}^{\infty}(Y)
\right)$ and define
\begin{align}
  \label{eq:128}
  \mu^{\varepsilon}(x)
  \coloneqq
  \nabla \left( \varphi^0(x) + \varepsilon \eta^1
  \left(x,\frac{x}{\varepsilon}\right) \right) + t \eta \left( x, \frac{x}{\varepsilon} \right),
\end{align}
then
\begin{align}
  \label{eq:130}
\mu^{\varepsilon} \wcv[2] \mu^0 \coloneqq \nabla \varphi^0(x) + \nabla_y
  \eta^1(x,y) + t \eta (x,y).
\end{align}
By monotonicity \ref{cond:a-monotonicity}, we obtain 
\begin{align*}
\int_{\Omega} \left( F^{\varepsilon} - \bfa
  \left(\frac{x}{\varepsilon}, \mu^{\varepsilon} \right) \right)
  \cdot \left( \nabla \varphi^{\varepsilon} - \mu^{\varepsilon}
  \right) \di x
  \ge 0, 
\end{align*}
or equivalently,
\begin{align*}
\int_{\Omega} \left( -\varphi^{\varepsilon}\Div F^{\varepsilon} - \bfa
  \left( \frac{x}{\varepsilon}, \mu^{\varepsilon} \right) \cdot \nabla
  \varphi^{\varepsilon}
  -F^{\varepsilon} \cdot \mu^{\varepsilon} + \bfa \left(
  \frac{x}{\varepsilon}, \mu^{\varepsilon} \right) \cdot \mu^{\varepsilon}
  \right) \di x
  \ge 0.
\end{align*}
We have $-\Div F^{\varepsilon} = f$ by \eqref{eq:p421} and, by  \ref{cond:a-continuity}
and \eqref{eq:128}, we have $\bfa
\left( \frac{x}{\varepsilon}, \mu^{\varepsilon} \right)$ and $\mu^{\varepsilon}$ are admissible in the sense of
\cref{def:admissible-function}. Therefore, letting $\varepsilon \to 0$, we obtain
\begin{align*}
  \frac{1}{\abs{Y}} \int_{\Omega \times Y} \left( f \varphi^0
  -\bfa (y, \mu^0) \cdot \left( \nabla \varphi^0 + \nabla_y \varphi^1
  \right) - F^0 \cdot \mu^0 + \bfa (y, \mu^0) \cdot \mu^0
  \right)
  \di x \di y \ge 0,
\end{align*}
or equivalently, 
\begin{align}
\label{eq:129}
  \frac{1}{\abs{Y}} \int_{\Omega \times Y} \left( f \varphi^0
  - F^0 \cdot \mu^0 + \bfa (y, \mu^0) \cdot \left(
  \nabla_y(\eta^1-\varphi^1) + t \eta \right)
  \right)
  \di x \di y \ge 0,
\end{align}
where $\mu^0$ is given by \eqref{eq:130}. In \eqref{eq:129}, choose a sequence $\eta^1$ that strongly
converges to $\varphi^1$ in $L^p \left( \Omega, W^{1,p}_{\per}(Y)/\RR
\right)$, we conclude that 
\begin{align}
  \label{eq:131}
  \begin{split}
&\frac{1}{\abs{Y}} \int_{\Omega \times Y}  f \varphi^0
  - F^0 \cdot \left( \nabla \varphi^0 + \nabla_y
   \varphi^1 + t \eta  \right) \di x \di y\\
  &\quad + \frac{1}{\abs{Y}} \int_{\Omega \times Y} \bfa (y, \nabla \varphi^0 + \nabla_y
  \varphi^1 + t \eta ) \cdot t \eta  
    \di x \di y \ge 0.
  \end{split}
\end{align}
By \eqref{eq:127} and \eqref{eq:126}, \eqref{eq:131} becomes
\begin{align*}
\frac{1}{\abs{Y}} \int_{\Omega \times Y}  
  - F^0 \cdot  t \eta +  \bfa (y, \nabla \varphi^0 + \nabla_y
  \varphi^1 + t \eta ) \cdot t \eta  
    \di x \di y \ge 0.
\end{align*}
Dividing both sides by $t > 0$, then letting $t \to 0$ and using
\ref{cond:a-continuity}, we obtain
\begin{align}
\label{eq:133}
\frac{1}{\abs{Y}} \int_{\Omega \times Y}  
  - F^0 \cdot  \eta +  \bfa (y, \nabla \varphi^0 + \nabla_y
  \varphi^1) \cdot  \eta  
    \di x \di y \ge 0,
\end{align}
for all $\eta \in \calD \left( \Omega, C_{\per}^{\infty}(Y) \right).$
Therefore, $F^0 = \bfa \left( y, \nabla \varphi^0 + \nabla_y \varphi^1
\right)$ and the two-scale limits $\varphi^0$
and $\varphi^1$ satisfy the
following system
\begin{subequations}
  \begin{align*}
    -\Div \left[ \frac{1}{\abs{Y}}\int_Y \bfa\left(y, \nabla \varphi^0 +
    \nabla_{\yy} \varphi^1 \right)\di \yy \right]
    &= f, \text{ in } \Omega,\\
    -\Div_{\yy}  \bfa \left(y, \nabla \varphi^0 + \nabla_{\yy}
    \varphi^1\right) 
    &=0, \text{ in } \Omega \times Y.
  \end{align*}
\end{subequations}
Since the system above has a unique solution by monotonicity
\ref{cond:a-monotonicity}, the entire sequence
$\varphi^{\varepsilon}$ is convergent. \hfill $\Box$

\section{Proof of \cref{sec:main-results-1}}
\label{sec:proof-crefs-results}
The proof adopts ideas from
\cite{dalmasoCorrectorsHomogenizationMonotone1990}, with improvements
by generalizing several estimates obtained in the cited paper and
 using two-scale convergence.

\subsection{Preliminary results}
\label{sec:prel-estim}
Let $p > 1 $. We establish some estimates that will be needed later in the proof of \cref{sec:main-results-1}.
\begin{lemma}
  \label{sec:some-inequalities-1}
  Let $Z$ be a metric space. There exists $C = C(d,p) > 0$ such that
  for any $\vv, \ww \in L^p(Z,\RR^d),$ we have
\begin{align}
  \label{eq:n80}
  \begin{split}
  &\norm{\vv - \ww}_{L^p(Z,\RR^d)}^p\\
  &\quad\le
    C \left(
    \int_{Z}  \left( 1 + \abs{\vv}^2 + \abs{\ww}^2
    \right)^{\frac{p-2}{2}}\abs{\vv -\ww}^2 \di z
    \right)^{\frac{1}{2}}\left(
    \abs{Z} +\norm{\vv}_{L^p(Z,\RR^d)}^p + \norm{\ww}_{L^p(Z,\RR^d)}^p
    \right)^{\frac{1}{2}}.
  \end{split}
\end{align}

\end{lemma}
\begin{proof}[Proof of \cref{sec:some-inequalities-1}]
  Let $\nu > 1$, $\mu \in [0,p]$, and $\lambda \in \RR$ be
parameters that will be chosen later, then using H\"{o}lder's inequality we obtain
\begin{align*}
  &\norm{\vv -\ww}_{L^p(Z,\RR^d)}^p\\ 
  &~= \int_{Z} {\abs{\vv -\ww}^{p-\mu}}{ \left( 1 +
    \abs{\vv}^2 + \abs{\ww}^{2} \right)^{-\lambda}} \cdot \abs{\vv-\ww}^{\mu}\left( 1 +
    \abs{\vv}^2 + \abs{\ww}^2 \right)^{\lambda} \di z
   \\
  &~\le
    \left(
    \int_{Z} \abs{\vv -\ww}^{\nu \left( p-\mu \right)} \left( 1 +
    \abs{\vv}^2 + \abs{\ww}^2 \right)^{-\lambda \nu} \di z
    \right)^{\frac{1}{\nu}}
    \left(
    \int_{Z} \abs{\vv -\ww}^{\frac{\mu \nu }{\nu -1}} \left(
    1+ \abs{\vv}^2 + \abs{\ww}^2 \right)^{\frac{\lambda \nu}{\nu-1}}
    \di z
    \right)^{\frac{\nu-1}{\nu}}.
\end{align*}
We want $\nu (p - \mu) =2$ and $-\lambda \nu = \frac{p-2}{2}$ (to
mimic \eqref{eq:n80}), so we let $\nu = 2,~\mu= p-1,$ and
$\lambda = \frac{2 -p}{4}$. With these choices, the estimate above becomes
\begin{align*}
  &\norm{\vv -\ww}_{L^p(Z,\RR^d)}^p\\ 
  &~\le
    \left(
    \int_{Z} \abs{\vv -\ww}^2 \left( 1 +
    \abs{\vv}^2 + \abs{\ww}^2 \right)^{\frac{p-2}{2}} \di z
    \right)^{\frac{1}{2}}
    \left(
    \int_{Z} \abs{\vv -\ww}^{2(p-1)} \left(
    1+ \abs{\vv}^2 + \abs{\ww}^2 \right)^{\frac{2-p}{2}}
    \di z
    \right)^{\frac{1}{2}}\\
  &~\le
    C\left(
    \int_{Z} \abs{\vv -\ww}^2 \left( 1 +
    \abs{\vv}^2 + \abs{\ww}^2 \right)^{\frac{p-2}{2}} \di z
    \right)^{\frac{1}{2}}
    \left(
    \int_{Z} \left( \abs{\vv}^2 + \abs{\ww}^2 \right)^{p-1} \left(
    1+ \abs{\vv}^2 + \abs{\ww}^2 \right)^{\frac{2-p}{2}}
    \di z
    \right)^{\frac{1}{2}}\\
  &~\le
    C\left(
    \int_{Z} \abs{\vv -\ww}^2 \left( 1 +
    \abs{\vv}^2 + \abs{\ww}^2 \right)^{\frac{p-2}{2}} \di z
    \right)^{\frac{1}{2}}
    \left(
    \int_{Z} \left( 1+ \abs{\vv}^2 + \abs{\ww}^2 \right)^{p-1} \left(
    1+ \abs{\vv}^2 + \abs{\ww}^2 \right)^{\frac{2-p}{2}}
    \di z
    \right)^{\frac{1}{2}}\\
  &~\le
    C\left(
    \int_{Z} \abs{\vv -\ww}^2 \left( 1 +
    \abs{\vv}^2 + \abs{\ww}^2 \right)^{\frac{p-2}{2}} \di z
    \right)^{\frac{1}{2}}
    \left(
    \int_{Z}  \left(
    1+ \abs{\vv}^2 + \abs{\ww}^2 \right)^{\frac{p}{2}}
    \di z
    \right)^{\frac{1}{2}}\\
  &~\le
    C\left(
    \int_{Z} \abs{\vv -\ww}^2 \left( 1 +
    \abs{\vv}^2 + \abs{\ww}^2 \right)^{\frac{p-2}{2}} \di z
    \right)^{\frac{1}{2}}
    \left(
    \abs{Z} +\norm{\vv}_{L^p(Z,\RR^d)}^p + \norm{\ww}_{L^p(Z,\RR^d)}^p
    \right)^{\frac{1}{2}},
\end{align*}
thus \eqref{eq:n80} is proved.
\end{proof}

Following Dal Maso and Defranceschi
\cite{dalmasoCorrectorsHomogenizationMonotone1990}, we start by
defining the function
\begin{align*}
M^{\varepsilon} \colon L^p(\Omega,\RR^d) \to L^p(\Omega,\RR^d)
\end{align*}
by 
\begin{align}
\label{eq:8}
  \left( M^{\varepsilon} \vv \right) (x)
  \coloneqq \sum_{i \in I^{\varepsilon}} \charfn_{Y^{\varepsilon}_i}
  (x) \, \frac{1}{\abs{Y^{\varepsilon}_i}} \int_{Y^{\varepsilon}_i}
  \vv(z) \di z,
\end{align}
where $\charfn_{Y^{\varepsilon}_i} (x)$ is the characteristic function
of the set $Y^{\varepsilon}_i$.  It can be shown that $M^{\varepsilon} \vv$ converges
to $\vv$ a.e. on $\Omega$ and strongly in $L^p$, cf. e.g., \cite[Chap. 6,
Prop. 9]{roydenRealAnalysis1988}, that is,
\begin{align}
\label{eq:21}
 \lim_{\varepsilon\to 0}\norm{M^{\varepsilon}\vv - \vv}_{L^p(\Omega,\RR^d)} 
  =0. 
\end{align}
Moreover, we also have
\begin{align*}
\norm{M^{\varepsilon} \vv}_{L^p (\Omega,\RR^d) } \le \norm{\vv}_{L^p (\Omega,\RR^d)}.
\end{align*}

Next, we define the function
\begin{align}
  \label{eq:20}
  \begin{split}
    &\bfp \colon Y \times \RR^d \to
      \RR^d,\qquad
      \bfp (y, \xi)
      \coloneqq \xi + \nabla_y \eta_{\xi} \left( y \right),
  \end{split}
\end{align}
where $\eta_{\xi}$ is the solution of \eqref{eq:13}.
It follows that
\begin{align}
\label{eq:67}
\int_Y \bfa \left( y, \bfp (y, \xi) \right) \cdot\bfp \left( y, \xi \right)
  \di y
  = \int_Y \bfa \left( y, \bfp \left( y, \xi \right) \right)\cdot \xi \di
  y,\qquad \text{ for all } \xi \in \RR^d.
\end{align}

We adapt an important corrector result by Dal Maso and Defranceschi
\cite[Theorem 2.1]{dalmasoCorrectorsHomogenizationMonotone1990} to our setting:

\begin{lemma}
\label{sec:preliminary-results}
We have
\begin{align}
\label{eq:113}
  \lim_{\varepsilon \to 0}\norm{\nabla \varphi^{\varepsilon}(x) - \bfp \left(\frac{x}{\varepsilon},
  M^{\varepsilon}(\nabla \varphi^0(x))\right)}_{L^p(\Omega,\RR^d)}
  = 0.
\end{align}
\end{lemma}

By comparing \eqref{eq:28} and \eqref{eq:13}, we obtain by uniqueness
that 
\begin{align}
\label{eq:109}
\bfp (y, \nabla \varphi^0(x)) = \nabla \varphi^0(x) + \nabla_y \varphi^1(x,y).
\end{align}

From \eqref{eq:28} and the fact that $\varphi^1$ is periodic with
respect to $y$, we obtain the following identity, similar to
\eqref{eq:67}, that holds for a.e. $x \in \Omega$,
\begin{align}
\label{eq:110}
\int_Y \bfa \left( y, \bfp(y,\nabla \varphi^0(x))\right) \cdot \bfp (y,
  \nabla \varphi^0(x)) \di y
  = \int_Y \bfa \left(y, \bfp(y, \nabla \varphi^0(x))\right) \cdot \nabla \varphi^0(x)
  \di y.
\end{align}

We will need the following generalized version of \cite[Lemma 3.2 and 3.4]{dalmasoCorrectorsHomogenizationMonotone1990}:

\begin{proposition} 
  \label{sec:prel-estim-1}
  There exists $C > 0$ depending on $d, p, \alpha,
\lambda_o, \Lambda_o, \Lambda_{*}$, which were introduced in Section
\ref{ss:setup}, such that for each $\varepsilon > 0$, we have
\begin{subequations}
  \label{eq:105}
  \begin{align}
    \label{eq:106}
    &\norm{\bfp \left( y,M^{\varepsilon} (\nabla \varphi^0)(x)
    \right)}_{L^p(Y_i^{\varepsilon}\times Y,\RR^d)}^p
      \le C \left( \abs{Y_i^{\varepsilon}\times Y} +\norm{M^{\varepsilon
      }(\nabla \varphi^0)}_{L^p(Y_i^{\varepsilon}\times Y,\RR^d)}^p \right),\\
    \label{eq:111}
    &\norm{\bfp \left( y,\nabla \varphi^0(x)
    \right)}_{L^p(Y_i^{\varepsilon}\times Y,\RR^d)}^p
      \le C \left( \abs{Y_i^{\varepsilon}\times Y} +\norm{\nabla
      \varphi^0}_{L^p(Y_i^{\varepsilon}\times Y,\RR^d)}^p \right),\\
    \label{eq:108}
    &\norm{\bfp\left( y,M^{\varepsilon}(\nabla \varphi^0)(x) \right) -
    \bfp \left( y,\nabla \varphi^0(x)
    \right)}_{L^p(Y_i^{\varepsilon}\times Y,\RR^d)}^p
       \\ \nonumber
    &\le C  \left( \abs{Y_i^{\varepsilon}} +
      \norm{M^{\varepsilon}(\nabla
      \varphi^0)}_{L^p(Y_i^{\varepsilon},\RR^d)}^p + \norm{\nabla \varphi^0}_{L^p(Y_i^{\varepsilon},\RR^d)}^p
    \right)^{\frac{2p-\alpha-1}{2p-\alpha}}
     \norm{M^{\varepsilon}(\nabla \varphi^0)-\nabla \varphi^0}_{L^p(Y_i^{\varepsilon},\RR^d)}^{\frac{p}{2p-\alpha}},
\end{align}
\end{subequations}
\end{proposition}

\begin{proof}
  \begin{enumerate}[wide]
\item \emph{Proof of \eqref{eq:106}}

  On $Y_i^{\varepsilon} \times Y$, note that
$M^{\varepsilon}(\nabla \varphi^0)(x)$ is independent of
  $x \in Y_i^{\varepsilon}$ and $y \in Y$ by definition. Thus, on $Y_i^{\varepsilon}$,  we let
  $\xi^i \coloneqq M^{\varepsilon}(\nabla \varphi^0(\cdot)) \in
  \RR^d$. By Fubini's Theorem, estimate
  \eqref{eq:106} can be written as 
\begin{align}
\label{eq:107}
\norm{\bfp \left( y,\xi^i)
    \right)}_{L^p(Y,\RR^d)}^p
      \le C \left( \abs{Y} + \abs{\xi^i}^p \right).
\end{align}

Thus we only need to prove \eqref{eq:107}. The proof of this inequality is similar to
\cite{dalmasoCorrectorsHomogenizationMonotone1990}, abeit our
assumption on $\bfa$ is different. Applying
\cref{sec:some-inequalities-1} and Young's inequality, we obtain
\begin{align*}
  \norm{\bfp \left( \cdot,\xi^{i} \right)}_{L^p(Y,\RR^d)}^p
  &\le C \left(
    \int_{Y} \abs{\bfp (\cdot,\xi^{i})}^2 \left( 1+
    \abs{\bfp(\cdot,\xi^{i})}^2 \right)^{\frac{p-2}{2}} \di y
    \right)^{\frac{1}{2}}
    \left( \abs{Y} + \norm{\bfp
    (\cdot,\xi^{i})}_{L^p(Y,\RR^d)}^p \right)^{\frac{1}{2}}\\
    &\le C \cdot 2C
    \int_{Y} \abs{\bfp (\cdot,\xi^{i})}^2 \left( 1+
      \abs{\bfp(\cdot,\xi^{i})}^2 \right)^{\frac{p-2}{2}} \di y
      +
    C \cdot \frac{1}{2C}\left( \abs{Y} + \norm{\bfp (\cdot,\xi^{i})}_{L^p(Y,\RR^d)}^p \right),
\end{align*}
so by rearranging terms and then using
monotonicity \ref{cond:a-monotonicity}, we have
\begin{align*}
  \norm{\bfp (\cdot,\xi^{i})}_{L^p(Y,\RR^d)}^p
  &\le C \left( \int_{Y} \abs{\bfp(\cdot,\xi^{i})}^2 \left( 1 +
    \abs{\bfp (\cdot,\xi^{i})}^2 \right)^{\frac{p-2}{2}}\di y + \abs{Y}\right) \\
  &\le C \left( \int_{Y} \left[ \bfa (y, \bfp(y, \xi^{i})) - \bfa
    (y,0)\right]\cdot \bfp(y,\xi^{i}) \di y  + \abs{Y}\right).
\end{align*}
By \eqref{eq:67}, the boundedness condition \ref{cond:a-bound}, and the continuity condition \ref{cond:a-continuity}, we obtain 
\begin{align*}
  \norm{\bfp (\cdot,\xi^{i})}_{L^p(Y,\RR^d)}^p
  &\le C \left(
    \int_Y \abs{\bfa (y, \bfp(y,\xi^{i})) \cdot \xi^{i}} \di y
    + \int_Y \abs{\bfa(y,0)\cdot \bfp (y,\xi^{i})} \di y + \abs{Y}
    \right)\\
  &\le C
    \left( \int_Y \left( 1 + \abs{\bfp
    (y,\xi^{i})}^2\right)^{\frac{p-1}{2}} \abs{\xi^{i}} \di y
    + \int_Y \abs{\bfp(\cdot,\xi^{i})} \di y + \abs{Y}
    \right).
\end{align*}
Let $\nu  > 0$ and $\mu > 0$ to be specified later. Applying Young's
inequality, we have
\begin{align*}
  \norm{\bfp (\cdot,\xi^{i})}_{L^p(Y,\RR^d)}^p
  &\le
    C \left(
    \int_Y   \nu^{\frac{p}{p-1}}\left( 1 +  \abs{\bfp(y,\xi^{i})}^2
    \right)^{\frac{p}{2}} \di y + \int_Y \frac{1}{ \nu^p}\abs{\xi^{i}}^p  \di y
  \right.\\
    &\quad+ \left.
    \int_Y  {\mu^p}\abs{\bfp(\cdot,\xi^{i})}^p \di y+ \int_Y\frac{1}{\mu^{\frac{p}{p-1}}} \di y
    +\abs{Y}
      \right)\\
   &\le
    C \left(
    \int_Y  \nu^{\frac{p}{p-1}}\left( 1 +  \abs{\bfp(y,\xi^{i})}^p
    \right)\di y + \int_Y\frac{1}{ \nu^p}\abs{\xi^{i}}^p  \di y
  \right.\\
    &\quad+ \left.
    \int_Y  {\mu^p}\abs{\bfp(\cdot,\xi^{i})}^p \di y+ \int_Y\frac{1}{\mu^{\frac{p}{p-1}}}  \di y
    +\abs{Y}
      \right) \\
  &\le C
    \left( \left( \nu^{\frac{p}{p-1}} + \mu^p \right)
    \norm{\bfp(\cdot,\xi^{i})}_{L^p(Y,\RR^d)}^p + \frac{1}{\nu^p}
    \norm{\xi^{i}}_{L^p(Y,\RR^d)}^p + \frac{\abs{Y}}{\mu^{\frac{p}{p-1}}} + \abs{Y} \right).
\end{align*}
Choose $\nu$ and $\mu$ small enough such that $\left(
  \nu^{\frac{p}{p-1}} + \mu^p \right) \le \frac{1}{2C}$, we conclude that 
\begin{align*}
  \norm{\bfp(\cdot,\xi^{i})}_{L^p(Y,\RR^d)}^p
  \le
  C \left( \abs{Y} + \norm{\xi^{i}}_{L^p(Y,\RR^d)}^p \right),
\end{align*}
which is  \eqref{eq:107}.

\item \emph{Proof of \eqref{eq:111}}
  
Applying \cref{sec:some-inequalities-1} and Young's
inequality, we obtain
\begin{align*}
  &\norm{\bfp \left( y,\nabla \varphi^0 (x) \right)}_{L^p(Y_i^{\varepsilon}\times Y,\RR^d)}^p\\
  &\le C \left(
    \int_{Y_i^{\varepsilon}\times Y} \abs{\bfp (y,\nabla \varphi^0 (x))}^2 \left( 1+
    \abs{\bfp(y,\nabla \varphi^0 (x))}^2 \right)^{\frac{p-2}{2}} \di x \di y
    \right)^{\frac{1}{2}}\\
    &\quad \cdot \left( \abs{Y_i^{\varepsilon}\times Y} + \norm{\bfp
    (y,\nabla \varphi^0 (x))}_{L^p(Y_i^{\varepsilon}\times Y,\RR^d)}^p \right)^{\frac{1}{2}}\\
    &\le C \cdot 2C
    \int_{Y_i^{\varepsilon}\times Y} \abs{\bfp (y,\nabla \varphi^0 (x))}^2 \left( 1+
      \abs{\bfp(y,\nabla \varphi^0 (x))}^2 \right)^{\frac{p-2}{2}} \di x \di y
  \\
  &\quad+
    C \cdot \frac{1}{2C}\left( \abs{Y_i^{\varepsilon}\times Y} + \norm{\bfp (y,\nabla \varphi^0 (x))}_{L^p(Y_i^{\varepsilon}\times Y,\RR^d)}^p \right),
\end{align*}
so by rearranging terms and then using
monotonicity \ref{cond:a-monotonicity}, we have
\begin{align*}
  &\norm{\bfp (y,\nabla \varphi^0 (x))}_{L^p(Y_i^{\varepsilon}\times Y,\RR^d)}^p\\
  &\le C \left( \int_{Y_i^{\varepsilon}\times Y} \abs{\bfp(y,\nabla \varphi^0 (x))}^2 \left( 1 +
    \abs{\bfp (y,\nabla \varphi^0 (x))}^2 \right)^{\frac{p-2}{2}}\di x \di y + \abs{Y_i^{\varepsilon}\times Y}\right) \\
  &\le C \left( \int_{Y_i^{\varepsilon}\times Y} \left[ \bfa (y, \bfp(y, \nabla \varphi^0 (x))) - \bfa
    (y,0)\right]\cdot \bfp(y,\nabla \varphi^0 (x)) \di x \di y  + \abs{Y_i^{\varepsilon}\times Y}\right).
\end{align*}
By \eqref{eq:110}, the boundedness condition \ref{cond:a-bound}, and the continuity condition \ref{cond:a-continuity}, we obtain 
\begin{align*}
  &\norm{\bfp (y,\nabla \varphi^0 (x))}_{L^p(Y_i^{\varepsilon}\times Y,\RR^d)}^p\\
  &\le C \left(
    \int_{Y_i^{\varepsilon}\times Y} \abs{\bfa (y, \bfp(y,\nabla \varphi^0 (x))) \cdot \nabla \varphi^0 (x)} \di x \di y
    + \int_{Y_i^{\varepsilon}\times Y} \abs{\bfa(y,0)\cdot \bfp (y,\nabla \varphi^0 (x))} \di x \di y + \abs{{Y_i^{\varepsilon}\times Y}}
    \right)\\
  &\le C
    \left( \int_{Y_i^{\varepsilon}\times Y} \left( 1 + \abs{\bfp
    (y,\nabla \varphi^0 (x))}^2\right)^{\frac{p-1}{2}} \abs{\nabla \varphi^0 (x)} \di x \di y
    + \int_{Y_i^{\varepsilon}\times Y} \abs{\bfp(y,\nabla \varphi^0 (x))} \di x \di y + \abs{{Y_i^{\varepsilon}\times Y}}
    \right).
\end{align*}
Let $\nu  > 0$ and $\mu > 0$ to be specified later. Applying Young's
inequality, we have
\begin{align*}
  &\norm{\bfp (y,\nabla \varphi^0 (x))}_{L^p({Y_i^{\varepsilon}\times Y},\RR^d)}^p\\
  &\le
    C \left(
    \int_{Y_i^{\varepsilon}\times Y}   \nu^{\frac{p}{p-1}}\left( 1 +  \abs{\bfp(y,\nabla \varphi^0 (x))}^2
    \right)^{\frac{p}{2}} \di x \di y + \int_{Y_i^{\varepsilon}\times Y} \frac{1}{ \nu^p}\abs{\nabla \varphi^0 (x)}^p  \di x \di y
  \right.\\
    &\quad+ \left.
    \int_{Y_i^{\varepsilon}\times Y}  {\mu^p}\abs{\bfp(y,\nabla \varphi^0 (x))}^p \di x \di y+ \int_{Y_i^{\varepsilon}\times Y}\frac{1}{\mu^{\frac{p}{p-1}}} \di x \di y
    +\abs{{Y_i^{\varepsilon}\times Y}}
      \right)\\
   &\le
    C \left(
    \int_{Y_i^{\varepsilon}\times Y}  \nu^{\frac{p}{p-1}}\left( 1 +  \abs{\bfp(y,\nabla \varphi^0 (x))}^p
    \right)\di x \di y + \int_{Y_i^{\varepsilon}\times Y}\frac{1}{ \nu^p}\abs{\nabla \varphi^0 (x)}^p  \di x \di y
  \right.\\
    &\quad+ \left.
    \int_{Y_i^{\varepsilon}\times Y}  {\mu^p}\abs{\bfp(y,\nabla \varphi^0 (x))}^p \di x \di y+ \int_{Y_i^{\varepsilon}\times Y}\frac{1}{\mu^{\frac{p}{p-1}}}  \di x \di y
    +\abs{{Y_i^{\varepsilon}\times Y}}
      \right) \\
  &\le C
    \left( \left( \nu^{\frac{p}{p-1}} + \mu^p \right)
    \norm{\bfp(y,\nabla \varphi^0 (x))}_{L^p({Y_i^{\varepsilon}\times Y},\RR^d)}^p + \frac{1}{\nu^p}
    \norm{\nabla \varphi^0 (x)}_{L^p({Y_i^{\varepsilon}\times Y},\RR^d)}^p + \frac{\abs{{Y_i^{\varepsilon}\times Y}}}{\mu^{\frac{p}{p-1}}} + \abs{{Y_i^{\varepsilon}\times Y}} \right).
\end{align*}
Choose $\nu$ and $\mu$ small enough such that $\left(
  \nu^{\frac{p}{p-1}} + \mu^p \right) \le \frac{1}{2C}$, we conclude that 
\begin{align*}
  \norm{\bfp(y,\nabla \varphi^0 (x))}_{L^p({Y_i^{\varepsilon}\times Y},\RR^d)}^p
  \le
  C \left( \abs{{Y_i^{\varepsilon}\times Y}} + \norm{\nabla \varphi^0 (x)}_{L^p({Y_i^{\varepsilon}\times Y},\RR^d)}^p \right),
\end{align*}
which is  \eqref{eq:111}.

\item \emph{Proof of \eqref{eq:108}}
  
By \cref{sec:some-inequalities-1}, we have 
\begin{align*}
&\norm{\bfp\left( y,M^{\varepsilon} (\nabla \varphi^0)(x) \right) -
    \bfp \left( y,\nabla \varphi^0 (x)
  \right)}_{L^p({Y_i^{\varepsilon}\times Y},\RR^d)}^p\\
  &\le
     C \left(
    \int_{{Y_i^{\varepsilon}\times Y}}  \left( 1 + \abs{\bfp\left( y,M^{\varepsilon} (\nabla \varphi^0)(x) \right)}^2 + \abs{\bfp\left( y,\nabla \varphi^0 (x) \right)}^2
    \right)^{\frac{p-2}{2}}\right.\\
  &\quad \quad \cdot \left. \vphantom{\int_{Y_i^{\varepsilon}\times Y}}\abs{\bfp\left( y,M^{\varepsilon} (\nabla \varphi^0)(x) \right) -\bfp\left( y,\nabla \varphi^0 (x) \right)}^2 \di x \di y
    \right)^{\frac{1}{2}}\\
  &\quad \cdot \left(
    1 +\norm{\bfp\left( y,M^{\varepsilon} (\nabla \varphi^0)(x) \right)}_{L^p({Y_i^{\varepsilon}\times Y},\RR^d)}^p + \norm{\bfp\left( y,\nabla \varphi^0 (x) \right)}_{L^p({Y_i^{\varepsilon}\times Y},\RR^d)}^p
    \right)^{\frac{1}{2}}.
\end{align*}
Using monotonicity \ref{cond:a-monotonicity}, \eqref{eq:106}, and
\eqref{eq:111}, we obtain
\begin{align}
  \label{eq:n81}
  \begin{split}
&\norm{\bfp\left( y,M^{\varepsilon} (\nabla \varphi^0)(x) \right) -
    \bfp \left( y,\nabla \varphi^0 (x)
  \right)}_{L^p({Y_i^{\varepsilon}\times Y},\RR^d)}^p\\
  &\le
    C \left( \int_{Y_i^{\varepsilon}\times Y}
    \left[ \bfa(y,\bfp(y, M^{\varepsilon} (\nabla \varphi^0)(x))) -
    \bfa (y,\bfp(y,\nabla \varphi^0 (x))) \right] \right.\\
    &\quad \quad
    \cdot \left. \vphantom{\int_{Y_i^{\varepsilon}\times Y}}\left( \bfp(y,M^{\varepsilon} (\nabla \varphi^0)(x)) - \bfp (y,\nabla \varphi^0 (x)) \right) \di x \di y
    \right)^{\frac{1}{2}}\\
  &\quad \cdot \left(
    \abs{{Y_i^{\varepsilon}\times Y}} + \norm{M^{\varepsilon} (\nabla \varphi^0)(x)}_{L^p({Y_i^{\varepsilon}\times Y},\RR^d)}^p + \norm{\nabla \varphi^0 (x)}_{L^p({Y_i^{\varepsilon}\times Y},\RR^d)}^p
    \right)^{\frac{1}{2}}\\
  &\le
    C \left( \int_{Y_i^{\varepsilon}\times Y}
    \left[ \bfa(y,\bfp(y, M^{\varepsilon} (\nabla \varphi^0)(x))) - \bfa (y,\bfp(y,\nabla \varphi^0 (x))) \right]
    \cdot \left( M^{\varepsilon} (\nabla \varphi^0)(x) - \nabla \varphi^0 (x) \right) \di x \di y
    \right)^{\frac{1}{2}}\\
  &\quad \cdot \left(
    \abs{{Y_i^{\varepsilon}\times Y}} + \norm{M^{\varepsilon} (\nabla \varphi^0)(x)}_{L^p({Y_i^{\varepsilon}\times Y},\RR^d)}^p + \norm{\nabla \varphi^0 (x)}_{L^p({Y_i^{\varepsilon}\times Y},\RR^d)}^p
    \right)^{\frac{1}{2}},
  \end{split}
\end{align}
where we use \eqref{eq:67} and \eqref{eq:110} in the last estimate. Observe that
by continuity \ref{cond:a-continuity} and H\"{o}lder's inequality
\begin{align*} 
  &\int_{Y_i^{\varepsilon}\times Y}
    \left[ \bfa(y,\bfp(y, M^{\varepsilon} (\nabla \varphi^0)(x))) - \bfa (y,\bfp(y,\nabla \varphi^0 (x))) \right]
    \cdot \left( M^{\varepsilon} (\nabla \varphi^0)(x) - \nabla \varphi^0 (x) \right) \di x \di y\\
  &~\le
    C \int_{Y_i^{\varepsilon}\times Y} \left( 1 + \abs{ \bfp (y,M^{\varepsilon} (\nabla \varphi^0)(x))}^2 +
    \abs{\bfp(y,\nabla \varphi^0 (x))}^2
    \right)^{\frac{p-1-\alpha}{2}}\\
  &~\quad \cdot 
    \abs{\bfp (y,M^{\varepsilon} (\nabla \varphi^0)(x)) -\bfp(y,\nabla \varphi^0 (x))}^{\alpha} \abs{M^{\varepsilon} (\nabla \varphi^0)(x) -
    \nabla \varphi^0 (x)}\di x \di y\\
  &~\le
    C \left( \int_{Y_i^{\varepsilon}\times Y} \left( 1 + \abs{ \bfp (y,M^{\varepsilon} (\nabla \varphi^0)(x))}^2 +
    \abs{\bfp(y,\nabla \varphi^0 (x))}^2 \right)^{\frac{p}{2}} \di x \di y
    \right)^{\frac{p-\alpha-1}{p}}\\
  &~\quad \cdot\left( \int_{Y_i^{\varepsilon}\times Y}
    \abs{\bfp (y,M^{\varepsilon} (\nabla \varphi^0)(x)) -\bfp(y,\nabla \varphi^0 (x))}^{p}
    \di x \di y \right)^{\frac{\alpha}{p}}
    \left( \int_{Y_i^{\varepsilon}\times Y}
    \abs{M^{\varepsilon} (\nabla \varphi^0)(x) -
    \nabla \varphi^0 (x)}^{p}\di x \di y \right)^{\frac{1}{p}}\\
  &~\le
    C \left( \int_{Y_i^{\varepsilon}\times Y} \left( 1 + \abs{ \bfp (y,M^{\varepsilon} (\nabla \varphi^0)(x))}^p +
    \abs{\bfp(y,\nabla \varphi^0 (x))}^p \right)\di x \di y
    \right)^{\frac{p-\alpha-1}{p}}\\
  &~\quad \cdot
    \norm{\bfp(y,M^{\varepsilon} (\nabla \varphi^0)(x)) -
    \bfp(y,\nabla \varphi^0 (x))}_{L^p(\Omega,\RR^d)}^{\alpha}
    \norm{M^{\varepsilon} (\nabla \varphi^0)(x)-\nabla \varphi^0 (x)}_{L^p({Y_i^{\varepsilon}\times Y},\RR^d)}\\
  &~\le
    C \left( \abs{{Y_i^{\varepsilon}\times Y}} + \norm{M^{\varepsilon} (\nabla \varphi^0)(x)}_{L^p({Y_i^{\varepsilon}\times Y},\RR^d)}^p + \norm{\nabla \varphi^0 (x)}_{L^p({Y_i^{\varepsilon}\times Y},\RR^d)}^p
    \right)^{\frac{p-\alpha-1}{p}}\\
    &~\quad \cdot \norm{\bfp(y,M^{\varepsilon} (\nabla \varphi^0)(x)) - \bfp(y,\nabla \varphi^0 (x))}_{L^p({Y_i^{\varepsilon}\times Y},\RR^d)}^{\alpha} \norm{M^{\varepsilon} (\nabla \varphi^0)(x)-\nabla \varphi^0 (x)}_{L^p({Y_i^{\varepsilon}\times Y},\RR^d)},
\end{align*}
where we also use \eqref{eq:106} and \eqref{eq:111} in the last inequality. Substituting
to \eqref{eq:n81}, we obtain 
\begin{align*}
&\norm{\bfp\left( y,M^{\varepsilon} (\nabla \varphi^0)(x) \right) -
    \bfp \left( y,\nabla \varphi^0 (x)
  \right)}_{L^p({Y_i^{\varepsilon}\times Y},\RR^d)}^p\\
  &\le
    C \left( \abs{{Y_i^{\varepsilon}\times Y}} + \norm{M^{\varepsilon} (\nabla \varphi^0)(x)}_{L^p({Y_i^{\varepsilon}\times Y},\RR^d)}^p + \norm{\nabla \varphi^0 (x)}_{L^p({Y_i^{\varepsilon}\times Y},\RR^d)}^p
    \right)^{\frac{p-\alpha-1}{2p}}\\
  &\quad \cdot
    \norm{\bfp(y,M^{\varepsilon} (\nabla \varphi^0)(x)) -
    \bfp(y,\nabla \varphi^0 (x))}_{L^p(\Omega,\RR^d)}^{\frac{\alpha}{2}}
    \norm{M^{\varepsilon} (\nabla \varphi^0)(x)-\nabla \varphi^0 (x)}_{L^p({Y_i^{\varepsilon}\times Y},\RR^d)}^{\frac{1}{2}}
    \\
&\quad
  \cdot \left( \abs{{Y_i^{\varepsilon}\times Y}} + \norm{M^{\varepsilon} (\nabla \varphi^0)(x)}_{L^p({Y_i^{\varepsilon}\times Y},\RR^d)}^p + \norm{\nabla \varphi^0 (x)}_{L^p({Y_i^{\varepsilon}\times Y},\RR^d)}^p \right)^{\frac{1}{2}},
\end{align*}
so
\begin{align*}
&\norm{\bfp\left( y,M^{\varepsilon} (\nabla \varphi^0)(x) \right) -
    \bfp \left( y,\nabla \varphi^0 (x)
  \right)}_{L^p({Y_i^{\varepsilon}\times Y},\RR^d)}^{\frac{2p-\alpha}{2}}\\
  &\le
    C\left( \abs{{Y_i^{\varepsilon}\times Y}} + \norm{M^{\varepsilon} (\nabla \varphi^0)(x)}_{L^p({Y_i^{\varepsilon}\times Y},\RR^d)}^p + \norm{\nabla \varphi^0 (x)}_{L^p({Y_i^{\varepsilon}\times Y},\RR^d)}^p
    \right)^{\frac{2p-\alpha-1}{2p}}\\
    &\quad \cdot \norm{M^{\varepsilon} (\nabla \varphi^0)(x)-\nabla \varphi^0 (x)}_{L^p({Y_i^{\varepsilon}\times Y},\RR^d)}^{\frac{1}{2}},
\end{align*}
which implies \eqref{eq:108}.
\end{enumerate}
\end{proof}

\begin{proposition}
  \label{sec:preliminary-results-3}
  Let $\Omega' \subset\subset \Omega$, there exists $C > 0$ depending
  on $d, p, \alpha, \lambda_o, \Lambda_o, \Lambda_{*},\diam(\Omega')$,
  which were introduced in Section \ref{ss:setup}, such that for each
  $\varepsilon > 0$, we have
\begin{align}
\label{eq:117}
    \begin{split}
    &\norm{\bfp\left( y,M^{\varepsilon}(\nabla \varphi^0)(x) \right) -
    \bfp \left( y,\nabla \varphi^0(x)
    \right)}_{L^p(  \Omega'\times Y,\RR^d)}^p
       \\ 
    &\quad\le C  \left( \abs{\Omega} +
      \norm{M^{\varepsilon}(\nabla
      \varphi^0)}_{L^p(\Omega,\RR^d)}^p +
      \norm{\nabla \varphi^0}_{L^p(\Omega,\RR^d)}^p
    \right)^{\frac{2p-\alpha-1}{2p-\alpha}}\\
     &\quad \quad \cdot\norm{M^{\varepsilon}(\nabla \varphi^0)-\nabla
      \varphi^0}_{L^p(\Omega,\RR^d)}^{\frac{p}{2p-\alpha}}.
    \end{split}
\end{align}
\end{proposition}
\begin{proof}
Let $\gamma \coloneqq \frac{1}{2p - \alpha}. $
For a fixed $\Omega' \subset\subset \Omega$, choose $\varepsilon > 0$ small
enough such that 
\begin{align*}
\Omega' \subset \bigcup_{i \in I^{\varepsilon}}
  Y_i^{\varepsilon}\subset \Omega,
\end{align*}
where $I^{\varepsilon}$ is defined in \eqref{eq:116}.
By \cref{sec:prel-estim-1}, 
 we obtain
\begin{align}
  \label{eq:82}
  \begin{split}
    &\norm{\bfp\left( y,M^{\varepsilon}(\nabla \varphi^0)(x) \right) -
    \bfp \left( y,\nabla \varphi^0(x)
    \right)}_{L^p(\Omega'\times Y,\RR^d)}^p\\
    &\le \norm{\bfp\left( y,M^{\varepsilon}(\nabla \varphi^0)(x) \right) -
    \bfp \left( y,\nabla \varphi^0(x)
    \right)}^p_{L^p
      \left(\bigcup_{i \in I^{\varepsilon}}
      Y_i^{\varepsilon},\RR^d\right)}
    \\
    &= \sum_{i \in I^{\varepsilon} } \norm{\bfp\left( y,M^{\varepsilon}(\nabla \varphi^0)(x) \right) -
    \bfp \left( y,\nabla \varphi^0(x)
    \right)}^p_{L^p
      (Y_i^{\varepsilon}\times Y,\RR^d)} \\
    &\le
      C \sum_{i \in I^{\varepsilon} }  \left(
      \abs{Y_i^{\varepsilon}} + \norm{M^{\varepsilon}(\nabla
      \varphi^0)}_{L^p(Y_i^{\varepsilon},\RR^d)}^p + \norm{\nabla \varphi^0}_{L^p(Y_i^{\varepsilon},\RR^d)}^p
      \right)^{1-\gamma}
      \norm{M^{\varepsilon}(\nabla \varphi^0)-\nabla
      \varphi^0}_{L^p(Y_i^{\varepsilon},\RR^d)}^{p\gamma}
  \end{split}
\end{align}

To estimate the last term, for $i\in
I^{\varepsilon}$, we define
\begin{align*}
  a_i
  &\coloneqq \abs{Y_i^{\varepsilon}} + \norm{M^{\varepsilon}(\nabla
      \varphi^0)}_{L^p(Y_i^{\varepsilon},\RR^d)}^p + \norm{\nabla
    \varphi^0}_{L^p(Y_i^{\varepsilon},\RR^d)}^p,\\
  b_i
  &\coloneqq
      \norm{M^{\varepsilon}(\nabla \varphi^0)-\nabla
      \varphi^0}_{L^p(Y_i^{\varepsilon},\RR^d)}^p,
\end{align*}
then
\begin{align}
  \label{eq:118}
  \begin{split}
  \sum_{i \in I^{\varepsilon} } a_i
  &\le \abs{\Omega} + \norm{M^{\varepsilon}(\nabla
      \varphi^0)}_{L^p(\Omega,\RR^d)}^p + \norm{\nabla
    \varphi^0}_{L^p(\Omega,\RR^d)}^p,\\
  \sum_{i \in I^{\varepsilon} } b_i
  &\le
      \norm{M^{\varepsilon}(\nabla \varphi^0)-\nabla
      \varphi^0}_{L^p(\Omega,\RR^d)}^p.
  \end{split}
\end{align}

Observe that $a_i,b_i \ge 0$ and $\gamma \in (0,1)$.  Let
$s \coloneqq \frac{1}{\gamma}$, $r \coloneqq \frac{1}{1-\gamma}$,
$c_i \coloneqq a_i^{1-\gamma}$, and $d_i \coloneqq b_i^{\gamma}$. Then
$\frac{1}{s} + \frac{1}{r} = 1$ and
by using H\"{o}lder's inequality with respect to the counting
measure, we obtain
\begin{align*}
\sum_{i \in I^{\varepsilon}} c_i d_i \le \left( \sum_{i \in I^{\varepsilon}} c_i^r
  \right)^{\frac{1}{r}} \left( \sum_{i \in I^{\varepsilon}} d_i^s \right)^{\frac{1}{s}},
\end{align*}
or equivalently, 
\begin{align}
\label{eq:86}
  \sum_{i \in I^{\varepsilon}} a_i^{1-\gamma}b_i^{\gamma}
  &\le \left( \sum_{i \in I^{\varepsilon}} a_i
    \right)^{1-\gamma} \left( \sum_{i \in I^{\varepsilon}} b_i \right)^{\gamma}.
\end{align}

From \eqref{eq:82}, \eqref{eq:118}, and \eqref{eq:86}, we conclude
\begin{align*}
  \begin{split}
&\norm{\bfp\left( y,M^{\varepsilon}(\nabla \varphi^0)(x) \right) -
    \bfp \left( y,\nabla \varphi^0(x)
    \right)}_{L^p(\Omega'\times Y,\RR^d)}^p\\
    &\le
      C \left( \abs{\Omega} + \norm{M^{\varepsilon}(\nabla
      \varphi^0)}_{L^p(\Omega,\RR^d)}^p + \norm{\nabla
      \varphi^0}_{L^p(\Omega,\RR^d)}^p \right)^{1-\gamma}
       \norm{M^{\varepsilon}(\nabla \varphi^0)-\nabla
      \varphi^0}_{L^p(\Omega,\RR^d)}^{p\gamma}.
  \end{split}
\end{align*}
\end{proof}

\begin{proposition}
\label{sec:preliminary-results-5}
We have 
\begin{align}
\label{eq:132}
  \lim_{\varepsilon \to 0}
  \norm{\bfp\left( y,M^{\varepsilon}(\nabla \varphi^0)(x) \right) -
    \bfp \left( y,\nabla \varphi^0(x)
  \right)}_{L^p(\Omega\times Y,\RR^d)}
  = 0.
\end{align}
\end{proposition}
\begin{proof}
  By monotonicity \ref{cond:a-monotonicity}, \eqref{eq:110}, and
  continuity \ref{cond:a-continuity}, we have for a.e. $x \in \Omega$,
\begin{align*}
  &\lambda_o \int_Y \abs{\bfp \left( y, \nabla \varphi^0(x) \right)}^p
    \di y\\
  &\le \int_Y \left[ \bfa \left( y,\bfp \left( y, \nabla \varphi^0(x)
    \right) \right) - \bfa (y,0) \right] \cdot \bfp \left( y, \nabla
    \varphi^0(x) \right) \di y\\
  &= \int_Y \bfa \left( y, \bfp \left( y, \nabla \varphi^0(x) \right)
    \right) \cdot \nabla \varphi^0 (x) \di y - \int_Y \bfa (y,0)
    \cdot \bfp \left( y, \nabla \varphi^0(x) \right) \di y\\
  &\le \int_Y \left[ \bfa \left( y, \bfp \left( y, \nabla \varphi^0(x) \right)
    \right) - \bfa (y,0)  \right] \cdot \nabla \varphi^0 (x) \di y
  \\
  &\quad{}+{} \int_Y \bfa (y,0) \cdot \nabla \varphi^0 (x) \di y - 
    \int_Y \bfa (y,0) \cdot \bfp \left( y, \nabla \varphi^0(x) \right)
    \di y\\
  &\le \int_Y \Lambda_o \left( 1 + \abs{\bfp \left( y, \nabla
    \varphi^0(x) \right)}^2 \right)^{\frac{p-1-\alpha}{2}} \abs{\bfp
    \left( y, \nabla \varphi^0(x) \right)}^{\alpha}
    \abs{\nabla \varphi^0 (x)} \di y
  \\
  &\quad{}+{} \int_Y \bfa (y,0) \cdot \nabla \varphi^0 (x) \di y - 
    \int_Y \bfa (y,0) \cdot \bfp \left( y, \nabla \varphi^0(x) \right)
    \di y.
\end{align*}
If $\abs{\bfp \left( y, \nabla \varphi^0(x) \right)} < 1$, then 
\begin{align*}
\left( 1 + \abs{\bfp \left( y, \nabla
    \varphi^0(x) \right)}^2 \right)^{\frac{p-1-\alpha}{2}} \abs{\bfp
  \left( y, \nabla \varphi^0(x) \right)}^{\alpha}
  \le C \abs{\bfp
    \left( y, \nabla \varphi^0(x) \right)}^{\alpha} \le C,
\end{align*}
and when $\abs{\bfp \left( y, \nabla \varphi^0(x) \right)}
\ge 1$, 
\begin{align*}
\left( 1 + \abs{\bfp \left( y, \nabla
    \varphi^0(x) \right)}^2 \right)^{\frac{p-1-\alpha}{2}} \abs{\bfp
  \left( y, \nabla \varphi^0(x) \right)}^{\alpha}
  \le C \abs{\bfp
    \left( y, \nabla \varphi^0(x) \right)}^{p-1}.
\end{align*}
Therefore, 
\begin{align*}
  &\lambda_o \int_Y \abs{\bfp \left( y, \nabla \varphi^0(x) \right)}^p
    \di y\\
  &\le C \int_Y  \left( 1 + \abs{\bfp \left( y, \nabla
    \varphi^0(x) \right)}^{{p-1}} \right)
    \abs{\nabla \varphi^0 (x)} \di y
  \\
  &\quad{}+{} \int_Y \bfa (y,0) \cdot \nabla \varphi^0 (x) \di y - 
    \int_Y \bfa (y,0) \cdot \bfp \left( y, \nabla \varphi^0(x) \right)
    \di y,
\end{align*}
and thus by boundedness \ref{cond:a-bound}, 
\begin{align*}
  &\lambda_o \int_Y \abs{\bfp \left( y, \nabla \varphi^0(x) \right)}^p
    \di y
  \\
  &\le C \left( \int_Y   \abs{\bfp \left( y, \nabla
    \varphi^0(x) \right)}^{{p-1}} 
    \abs{\nabla \varphi^0 (x)} \di y
    +  \abs{\nabla \varphi^0(x)} + 
    \int_Y \abs{ \bfp \left( y, \nabla \varphi^0(x) \right)} \di y \right).
\end{align*}
Applying Young inequality for the first and the last integrands, we
obtain for a.e. $x \in \Omega$,
\begin{align}
\label{eq:134}
\int_Y \abs{\bfp \left( y, \nabla \varphi^0(x) \right)}^p
  \di y
  \le
  C_{*} \left( \abs{\nabla \varphi^0(x)}^p + \abs{\nabla \varphi^0(x)} + 1 \right),
\end{align}
for some constant $C_{*} > 0$ independent of $\varepsilon > 0$.

In
  \eqref{eq:13}, consider the special case
  $\xi =0$, then
\begin{align}
\label{eq:123}
-\Div_y \bfa (y, \nabla_{y} \eta _0) = 0, \quad \eta_0 \in W_{\per}^{1,p}(Y)/\RR
\end{align}
  By
  \ref{cond:a-bound}, \ref{cond:a-monotonicity}, \eqref{eq:123}, and H\"{o}lder
  inequality we obtain
\begin{align*}
  \lambda_o \norm{\nabla_y \eta_0}_{L^p}^p
  &\le \int_{Y} \left[  \bfa \left( y, \nabla_y
    \eta_0  \right) - \bfa \left(
    y,0 \right) \right] \cdot \nabla_y
    \eta_0 \di y
    = - \int_{Y}\bfa \left(
    y,0 \right) \cdot \nabla_y
    \eta_0 \di y
    \le
    C\Lambda_{*}
      \norm{\nabla_y \eta_0}_{L^p(Y)}. 
\end{align*}
Therefore,
\begin{align}
\label{eq:122}
 \int_Y \abs{\bfp (y,0)}^p \di y  = \norm{\nabla_y\eta_0}_{L^p(Y,\RR^d)}^p \le C
(d,p,\lambda_o, \Lambda_{*}) \eqqcolon C_{**}.
\end{align}

Fix $\delta > 0$. Since $\varphi^0 \in W^{1,p}(\Omega)$, there exists
$\varepsilon_{\delta} > 0$ such that whenever $\varepsilon \in (0,
\varepsilon_{\delta}),$ we have 
\begin{align}
\label{eq:135}
2^p \left( C_{*} \int_{\bigcup_{i\in J^{\varepsilon}} Z_i^{\varepsilon}} \left(
  \abs{\nabla \varphi^0(x)}^p + \abs{\nabla \varphi^0(x)} + 1 \right)
  \di x
   + C_{**} \abs{\bigcup_{i\in J^{\varepsilon}} Z_i^{\varepsilon}}
   \right)
  \le \delta.
\end{align}

In \cref{sec:preliminary-results-3}, let
$\Omega' = \left( \bigcup_{i \in I^{\varepsilon}} Y_i^{\varepsilon}
\right),$ we have
\begin{align*}
  &\norm{\bfp\left( y,M^{\varepsilon}(\nabla \varphi^0)(x) \right) -
    \bfp \left( y,\nabla \varphi^0(x)
    \right)}_{L^p(\Omega\times Y,\RR^d)}^p
  \\
  &= \norm{\bfp\left( y,M^{\varepsilon}(\nabla \varphi^0)(x) \right) -
    \bfp \left( y,\nabla \varphi^0(x)
    \right)}_{L^p \left(\Omega' \times Y,\RR^d\right)}^p\\
  &\qquad + \int_{\left( \bigcup_{i \in J^{\varepsilon}}
    Z_i^{\varepsilon} \right) \times Y} \abs{\bfp\left( y,M^{\varepsilon}(\nabla \varphi^0)(x) \right) -
    \bfp \left( y,\nabla \varphi^0(x)
    \right)}^p \di x \di y
  \\
  &\le C  \left( \abs{\Omega} +
    \norm{M^{\varepsilon}(\nabla
    \varphi^0)}_{L^p(\Omega,\RR^d)}^p +
    \norm{\nabla \varphi^0}_{L^p(\Omega,\RR^d)}^p
    \right)^{\frac{2p-\alpha-1}{2p-\alpha}}\cdot\norm{M^{\varepsilon}(\nabla \varphi^0)-\nabla
    \varphi^0}_{L^p(\Omega,\RR^d)}^{\frac{p}{2p-\alpha}}\\
  &\qquad + 2^p \int_{\left( \bigcup_{i \in J^{\varepsilon}}
    Z_i^{\varepsilon} \right) \times Y} \abs{\bfp\left(
    y,0) \right)}^p + \abs{
    \bfp \left( y,\nabla \varphi^0(x)
    \right)}^p \di x \di y
  \\
  &\le C  \left( \abs{\Omega} +
    \norm{M^{\varepsilon}(\nabla
    \varphi^0)}_{L^p(\Omega,\RR^d)}^p +
    \norm{\nabla \varphi^0}_{L^p(\Omega,\RR^d)}^p
    \right)^{\frac{2p-\alpha-1}{2p-\alpha}}\cdot\norm{M^{\varepsilon}(\nabla \varphi^0)-\nabla
    \varphi^0}_{L^p(\Omega,\RR^d)}^{\frac{p}{2p-\alpha}}\\
  &\qquad + 2^p \int_{\bigcup_{i \in J^{\varepsilon}}
    Z_i^{\varepsilon}} \left(  \int_Y \abs{\bfp\left(
    y,0) \right)}^p \di y + \int_Y \abs{
    \bfp \left( y,\nabla \varphi^0(x)
    \right)}^p  \di y  \right)  \di x 
  \\
  &\le C  \left( \abs{\Omega} +
    \norm{M^{\varepsilon}(\nabla
    \varphi^0)}_{L^p(\Omega,\RR^d)}^p +
    \norm{\nabla \varphi^0}_{L^p(\Omega,\RR^d)}^p
    \right)^{\frac{2p-\alpha-1}{2p-\alpha}}\cdot\norm{M^{\varepsilon}(\nabla \varphi^0)-\nabla
    \varphi^0}_{L^p(\Omega,\RR^d)}^{\frac{p}{2p-\alpha}} + \delta,
\end{align*}
where we have used Fubini theorem, \eqref{eq:134}, \eqref{eq:122}, and
\eqref{eq:135} in the last estimate. Letting $\varepsilon \cv 0$, we
obtain
\begin{align*}
  \lim_{\varepsilon \to 0}
  \norm{\bfp\left( y,M^{\varepsilon}(\nabla \varphi^0)(x) \right) -
    \bfp \left( y,\nabla \varphi^0(x)
  \right)}_{L^p(\Omega\times Y,\RR^d)}^p
  \le \delta,
\end{align*}
since we have $M^{\varepsilon}(\nabla \varphi^0) \cv \nabla \varphi^0$
in $L^p$-norm, see \eqref{eq:21}. Because $\delta > 0$ is arbitrary,
we conclude that
\begin{align*}
  \lim_{\varepsilon \to 0}
  \norm{\bfp\left( y,M^{\varepsilon}(\nabla \varphi^0)(x) \right) -
    \bfp \left( y,\nabla \varphi^0(x)
  \right)}_{L^p(\Omega\times Y,\RR^d)}^p
  =0.
\end{align*}
\end{proof}

\begin{proposition}
  \label{sec:preliminary-results-4}
  We have
\begin{align}
  \label{eq:119}
  \lim_{\varepsilon\to 0} \left( 
\norm{\bfp \left(\frac{x}{\varepsilon},
  M^{\varepsilon}(\nabla \varphi^0)(x)\right)}_{L^p(\Omega,\RR^d)}^p
  - \norm{\bfp \left(y,
  M^{\varepsilon}(\nabla \varphi^0)(x)\right)}_{L^p(\Omega \times
  Y,\RR^d)}^p  \right) = 0.
\end{align}
\end{proposition}
\begin{proof}
  For $\varepsilon > 0$, let $y = \frac{x}{\varepsilon}-i$ and
  $Z_i = \varepsilon^{-1}(Z_i^{\varepsilon}-i).$ Note that $\bfp$ is
  periodic with respect to the $y$ variable and $Z_i^{\varepsilon} =
  Y_i^{\varepsilon}$, $Z_i=Y$ whenever $i \in I^{\varepsilon}$. On $Y_i^{\varepsilon} \times Y$, note that
$M^{\varepsilon}(\nabla \varphi^0)(x)$ is independent of
  $x \in Y_i^{\varepsilon}$ and $y \in Y$ by definition. Thus, on $Y_i^{\varepsilon}$,  we let
  $\xi^i \coloneqq M^{\varepsilon}(\nabla \varphi^0(\cdot)) \in
  \RR^d$.
  We have
\begin{align*}
\int_{\Omega} \abs{\bfp \left( \frac{x}{\varepsilon}, M^{\varepsilon}
  \left( \nabla \varphi^0 \right)(x) \right)}^p \di x
  &= \sum_{i \in I^{\varepsilon} \cup J^{\varepsilon}}
    \int_{Z_i^{\varepsilon}} \abs{\bfp \left( \frac{x}{\varepsilon},
    M^{\varepsilon} \left( \nabla \varphi^0 \right)(x) \right)}^p\di x
  \\
  &=  \sum_{i \in I^{\varepsilon}}
    \int_{Z_i^{\varepsilon}} \abs{\bfp \left( \frac{x}{\varepsilon},
    \xi^i \right)}^p \di x
   +  \sum_{i \in  J^{\varepsilon}}
    \int_{Z_i^{\varepsilon}} \abs{\bfp \left( \frac{x}{\varepsilon},
    0 \right)}^p \di x
    \\
  &=  \sum_{i \in I^{\varepsilon}}
    \int_{Y_i^{\varepsilon}} \abs{\bfp \left( \frac{x}{\varepsilon},
    \xi^i \right)}^p \di x
   +  \sum_{i \in  J^{\varepsilon}}
    \int_{Z_i^{\varepsilon}} \abs{\bfp \left( \frac{x}{\varepsilon},
    0 \right)}^p \di x\\
  &=  \sum_{i \in I^{\varepsilon} } \int_{Y}
    \abs{\bfp (y,\xi^i)} ^p \di y~ \varepsilon^d
  +  \sum_{i \in  J^{\varepsilon}}
    \int_{Z_i} \abs{\bfp \left( y,
    0 \right)}^p \di y~ \varepsilon^d \\
  &=  \sum_{i \in I^{\varepsilon} }
    \int_{Y_i^{\varepsilon}\times Y} \abs{\bfp \left( y, \xi^i
    \right)}^p \di x \di y
  +  \sum_{i \in  J^{\varepsilon}}
    \int_{Z_i} \abs{\bfp \left( y,
    0 \right)}^p \di y~ \varepsilon^d\\
  &=  \sum_{i \in I^{\varepsilon}}
    \int_{Y_i^{\varepsilon}\times Y} \abs{\bfp(y,
    M^{\varepsilon}(\nabla \varphi^0)(x))}^p \di x\di y
  + \sum_{i \in J^{\varepsilon}}
    \int_{Z_i^{\varepsilon}\times Y} \abs{\bfp(y,
    0)}^p \di x\di y \\
  &\qquad +  \sum_{i \in  J^{\varepsilon}}
    \int_{Z_i} \abs{\bfp \left( y,
    0 \right)}^p \di y~ \varepsilon^d -  \sum_{i \in J^{\varepsilon}}
    \int_{Z_i^{\varepsilon}\times Y} \abs{\bfp(y,
    0)}^p \di x\di y\\
  &= \int_{\Omega \times Y} \abs{\bfp(y,
    M^{\varepsilon}(\nabla \varphi^0)(x))}^p \di x\di y\\
  &\qquad +  \sum_{i \in  J^{\varepsilon}}
    \int_{Z_i} \abs{\bfp \left( y,
    0 \right)}^p \di y~ \varepsilon^d -  \sum_{i \in J^{\varepsilon}}
    \int_{Z_i^{\varepsilon}\times Y} \abs{\bfp(y,
    0)}^p \di x\di y.
\end{align*}
Therefore,
\begin{align*}
  &\abs{\norm{\bfp \left(\frac{x}{\varepsilon},
    M^{\varepsilon}(\nabla \varphi^0)(x)\right)}_{L^p(\Omega,\RR^d)}^p
    - \norm{\bfp \left(y,
    M^{\varepsilon}(\nabla \varphi^0)(x)\right)}_{L^p(\Omega \times
    Y,\RR^d)}^p}\\
  &\le \abs{\sum_{i \in  J^{\varepsilon}}
    \int_{Z_i} \abs{\bfp \left( y,
    0 \right)}^p \di y~ \varepsilon^d -  \sum_{i \in J^{\varepsilon}}
    \int_{Z_i^{\varepsilon}\times Y} \abs{\bfp(y,
    0)}^p \di x\di y}\\
  &\le \sum_{i \in  J^{\varepsilon}}
    \int_{Y} \abs{\bfp \left( y,
    0 \right)}^p \di y~ \varepsilon^d +  \sum_{i \in J^{\varepsilon}}
    \int_{Y} \abs{\bfp(y,
    0)}^p \di y~\varepsilon^d \abs{Z_i}\\
  &\le 2 \sum_{i \in  J^{\varepsilon}}
    \int_{Y} \abs{\bfp \left( y,
    0 \right)}^p \di y~ \varepsilon^d\\
  &\le 2 \varepsilon^d\sum_{i \in J^{\varepsilon}}^{} C \qquad \text{(by
    \eqref{eq:122})}\\
  &\le 2C \varepsilon^d \cdot \frac{\abs{\partial
    \Omega}}{\varepsilon^{d-1}}\\
  &\le C \varepsilon \quad\cv 0 \text{ as } \varepsilon \to 0.
\end{align*}


\end{proof}

\begin{proposition}[Characterization of strong two-scale convergence]
\label{sec:preliminary-results-1}
Let $1 < p < \infty$. Define the \emph{two-scale composition function} by 
\begin{align*}
  S^{\varepsilon} (x,y)
  \coloneqq \varepsilon \left[ \frac{x}{\varepsilon} \right] +
  \varepsilon y, \qquad (x, y) \in \Omega \times Y.
\end{align*}
A sequence $\left\{ v^{\varepsilon} \right\}_{\varepsilon>0} \subset L^p(\Omega)$
strongly two-scale converges to $v \in L^p(\Omega \times Y)$ if and
only if 
\begin{align}
\label{eq:75}
v^{\varepsilon} \circ S^{\varepsilon} \cv v \quad \text{ in
  }L^p(\Omega\times Y).
\end{align} 
\end{proposition}

\begin{proof}
  The sufficient condition
  is proved in \cite[Proposition
  2.7]{visintinTwoscaleCalculus2006}. 

To prove the necessary condition, we assume that $v^{\varepsilon}$ strongly two-scale converges to
  $v$, i.e., $v^{\varepsilon} \wcv[2] v$ in $L^p(\Omega \times Y)$ and
  $\norm{v^{\varepsilon}}_{L^p(\Omega)} \cv \norm{v}_{L^p(\Omega
    \times Y)}.$ On the one hand, \cite[Proposition
  2.5]{visintinTwoscaleCalculus2006} implies that
  $v^{\varepsilon} \circ S^{\varepsilon} \wcv v$ in
  $L^p(\Omega \times Y)$. On the other hand, \cite[Lemma
  1.1]{visintinTwoscaleCalculus2006} implies that
  $\norm{v^{\varepsilon} \circ S^{\varepsilon}}_{L^p(\Omega \times Y)}
  = \norm{v^{\varepsilon}}_{L^p(\Omega)}$, hence $\norm{v^{\varepsilon}
    \circ S^{\varepsilon}}_{L^p(\Omega \times Y)} \cv
  \norm{v}_{L^p(\Omega \times Y)}$. By the result in real analysis
  \cite[Page 124]{brezisFunctionalAnalysisSobolev2011}, we conclude 
\begin{align*}
v^{\varepsilon} \circ S^{\varepsilon} \cv v \qquad \text{ in }
  L^p(\Omega \times Y).
\end{align*}
\end{proof}

There are two different definitions for strong two-scale convergence:
we use the one in \cref{def:two-scale}, while author of
\cite{visintinTwoscaleCalculus2006} uses \eqref{eq:75} as the
definition of strong two-scale
convergence. \cref{sec:preliminary-results-1} shows that the two
definitions are equivalent when $p > 1$. The necessary direction does
not hold when $p = 1.$

\subsection{Proof of \eqref{eq:112}, \eqref{eq:44}, and \eqref{eq:23}}
\label{sec:proof-eqrefeq:5}

From \cref{sec:preliminary-results}, \cref{sec:preliminary-results-5},
and \cref{sec:preliminary-results-4}, we conclude that 
\begin{align}
\label{eq:120}
\lim_{\varepsilon \to 0} \norm{\nabla
  \varphi^{\varepsilon}}_{L^p(\Omega,\RR^d)}
  = \norm{\bfp \left(y,
  \nabla \varphi^0(x)\right)}_{L^p(\Omega \times
  Y,\RR^d)}
  = \norm{\nabla \varphi^0(x) + \nabla_y \varphi^1(x,y)}_{L^p(\Omega \times
  Y,\RR^d)}.
\end{align}
Using this with \eqref{eq:4}, we conclude that
\begin{align*}
\nabla \varphi^{\varepsilon}(x) \cv[2] \nabla \varphi^0(x) + \nabla_y
  \varphi^1(x,y) \qquad \text{ in }
L^p(\Omega \times Y).
\end{align*}
The above strong two-scale convergence, however, does not immediately imply
\eqref{eq:112} because \eqref{eq:120} does not guarantee 
\begin{align*}
\lim_{\varepsilon \to 0} \norm{\frac{\partial
  \varphi^{\varepsilon}}{\partial x_i}}_{L^p(\Omega)}
  =  \norm{\frac{\partial \varphi^0}{\partial x_i} + \frac{\partial
  \varphi^1}{\partial y_i}}_{L^p(\Omega \times Y)}.
      \end{align*}
By \cref{sec:preliminary-results-1}, we have 
$(\nabla \varphi^{\varepsilon}) \circ S^{\varepsilon} \cv \nabla
\varphi^0 + \nabla_y \varphi^1$ in $L^p(\Omega \times Y, \RR^d)$, so
$\frac{\partial \varphi^{\varepsilon}}{ \partial x_i} \circ
S^{\varepsilon} \cv \frac{\partial \varphi^0}{\partial x_i} +
\frac{\partial\varphi^1}{\partial y_i}$ in $L^p(\Omega \times Y)$, and
thus, $\frac{\partial \varphi^{\varepsilon}}{ \partial x_i} \cv[2] \frac{\partial \varphi^0}{\partial x_i} +
\frac{\partial\varphi^1}{\partial y_i}$ in $L^p(\Omega \times Y)$,
which is \eqref{eq:112}. Identities  \eqref{eq:44} and \eqref{eq:23} now follows from \cref{sec:two-scale-conv-5}.

\hfill$\Box$

\section{Proof of \cref{thm:main}}
\label{sec:proof-crefthm:main}

The weak convergence of $\varphi^{\varepsilon} $ to $\varphi^{0}$ in $W_0^{1,p}(\Omega)$ follows from \eqref{eq:32}, together with the homogenized equation \eqref{eq:31} for the electrostatic problem, and results obtained in previous sections. Therefore, at this point, it only remains to deal with the homogenization of
the elastic equation \eqref{eq:r7}.  The roadmap to obtain this homogenization result is as follows. In \cref{sec:bound-elast-displ}, we demonstrate that the sequence of
solutions $\{ \uu^{\varepsilon} \}_{\varepsilon >0}$ is uniformly
bounded in $\mathrm{BV}(\Omega',\RR^d)$, with $\Omega'\subset\subset\Omega$,
which will be used later for the two-scale convergence
argument. Then, we discuss additional regularity for the solution
$\uu^{\varepsilon} \in W_{\loc}^{1,1}(\Omega,\RR^d)$ of
\eqref{eq:p410}-\eqref{eq:p415} and derive an a priori estimate for $\uu^{\varepsilon}$. After
that, 
 we discuss existence of such a solution
$\uu^{\varepsilon}$. 
And
finally, in \cref{sec:two-scale-homog}, we adapt two-scale convergence
to derive the homogenization system for
$ \uu^{\varepsilon}$. The solution of the obtained two-scale homogenized problem exhibits additional
regularity of $W^{1,1},$
which allows one to explicitly write the cell and effective systems.


\subsection{An a priori estimate}
\label{sec:bound-elast-displ}

We now prove the boundedness of the sequence of elastic displacements
$\uu^{\varepsilon}$, which is the solution of
\eqref{eq:p410}-\eqref{eq:p415}, in the $\mathrm{BV}-$norm. Suppose $\Omega'
\subset\subset \Omega$ with $C^{1,\kappa}-$boundary, for some $0<\kappa<1$. 
Consider the following auxiliary problem:\\ For given functions
$\bg' \in L^2(\Omega, \RR^d)$ and
$\ff'\in L^2(\Omega, \RR^{d \times d})$, let
$\ww^{\varepsilon} \in W_0^{1,2}(\Omega,\RR^d)$ be the solution of
\begin{align}
\label{eq:89}
\int_{\Omega} \bfB \left( \frac{x}{\varepsilon} \right) \, \DD(\ww^{\varepsilon}) : \DD(\uu)
  \di x 
  = \int_{\Omega} \bg' \cdot \uu \di x
  - \int_{\Omega} \ff' : \DD(\uu) \di x
\end{align}
for all $\uu \in \calD (\Omega,\RR^d)$. This solution
$\ww^{\varepsilon}\in W^{1,2}_0(\Omega,\RR^d)$ exists and is unique
by the Lax-Milgram theorem 
and
\begin{align}
\label{eq:29}
  \norm{\DD \ww^{\varepsilon}}_{L^2(\Omega,\RR^{d\times d})}
  \le C_2 \left( \norm{\bg'}_{L^2 (\Omega,\RR^{d})} +
  \norm{\ff'}_{L^2 (\Omega,\RR^{d\times d})} \right), 
\end{align}
for some constant $C_2 > 0$ independent of $\varepsilon > 0$. Then,
Korn's inequality implies 
\begin{align}
\label{eq:91}
  \norm{\ww^{\varepsilon}}_{W^{1,2}(\Omega,\RR^{d})}
  \le C_2 \left( \norm{\bg'}_{L^2 (\Omega,\RR^{d})} +
  \norm{\ff'}_{L^2 (\Omega,\RR^{d\times d})} \right).
\end{align}

We now consider the problem \eqref{eq:89} with
$\bg' \in L^\infty(\Omega, \RR^d)$ and
$\ff' \in L^\infty(\Omega,\RR^{d \times d})$, then by  \cite[Theorem
1.1 and 1.9]{liEstimatesEllipticSystems2003} (see also
\cite{avellanedaCompactnessMethodsTheory1989,avellanedaCompactnessMethodsTheory1987,shenPeriodicHomogenizationElliptic2018,dangGlobalGradientEstimate2022,prangeUniformEstimatesHomogenization2014}), $\DD\ww^{\varepsilon}$ is
piecewise continuous, and the following Lipschitz estimate holds:
\begin{align}
  \label{eq:7}
  \begin{split}
    \norm{\ww^{\varepsilon}}_{W^{1,\infty}(\Omega',\RR^d)}
    &\le C_1 \left( \norm{\ww^{\varepsilon}}_{L^2(\Omega,\RR^d)} + \norm{\bg'}_{L^{\infty}(\Omega,\RR^{d})} +
      \norm{\ff'}_{L^{\infty}(\Omega,\RR^{d\times d})} \right)\\
    &\le C_1 \left( \norm{\bg'}_{L^{\infty}(\Omega,\RR^{d})} +
      \norm{\ff'}_{L^{\infty}(\Omega,\RR^{d\times d})} \right),
  \end{split}
\end{align}
for some $C_1 = C_1(d,\kappa,\lambda_e,\Lambda_e, \Omega', \Omega)> 0$
independent of $\varepsilon > 0$. In the last estimate, we also use
\eqref{eq:91} and the fact that $\Omega$ is bounded.

By \eqref{eq:32} and an extension of H\"{o}lder inequality \cite[Remark
2, p. 93]{brezisFunctionalAnalysisSobolev2011}, we obtain 
\begin{align}
  \label{eq:92}  \norm{\maxss^{\varepsilon}}_{L^{p/2}} = \norm{\nabla
  \varphi^{\varepsilon}\otimes \nabla \varphi^{\varepsilon}}_{L^{p/2}}
  &\le \norm{\nabla \varphi^{\varepsilon}}_{L^p} \norm{\nabla
    \varphi^{\varepsilon}}_{L^p} \le C \norm{f}_{L^{p'}}^2.
\end{align}

Consider the problem \eqref{eq:89} with
$\bg' \in \calD(\Omega',\RR^d), \ff' \in \calD(\Omega',\RR^{d \times
  d})$, then $\ww^{\varepsilon} = 0$ on $\Omega \setminus \Omega'$.
From \eqref{eq:r7}, \eqref{eq:89}, \eqref{eq:7}, and \eqref{eq:92}, we
have
\begin{align}
\label{eq:90}
  \begin{split}
    &\abs{\int_{\Omega'} \bg' \cdot \uu^{\varepsilon} \di x
    - \int_{\Omega'} \ff' : \DD(\uu^{\varepsilon}) \di x}\\
    &=\abs{\int_{\Omega} \bg' \cdot \uu^{\varepsilon} \di x
  - \int_{\Omega} \ff' : \DD(\uu^{\varepsilon}) \di x}\\
    &= \abs{\int_{\Omega} \bfB \left( \frac{x}{\varepsilon} \right)\DD(\ww^{\varepsilon}) : \DD(\uu^{\varepsilon})
      \di x} \\
    &=\abs{\int_{\Omega} \bg \cdot \ww^{\varepsilon} \di x
  - \int_{\Omega} \bfC \left( \frac{x}{\varepsilon} \right)\maxss^{\varepsilon}: \DD(\ww^{\varepsilon})\di x} 
  \\
    &= \abs{\int_{\Omega'} \bg \cdot \ww^{\varepsilon} \di x
  - \int_{\Omega'} \bfC \left( \frac{x}{\varepsilon} \right)\maxss^{\varepsilon}: \DD(\ww^{\varepsilon})\di x}
  \\
    &\le C \left( \norm{\bg'}_{L^{\infty}(\Omega',\RR^d)} +
      \norm{\ff'}_{L^{\infty}(\Omega',\RR^{d\times d})}\right) \left(
      \norm{\bg}_{L^1(\Omega',\RR^d)} +
      \norm{\maxss^{\varepsilon}}_{L^1(\Omega',\RR^{d\times
      d})}\right)\\
      &\le C \left( \norm{\bg'}_{L^{\infty}(\Omega',\RR^d)} +
      \norm{\ff'}_{L^{\infty}(\Omega',\RR^{d\times d})}\right) \left(
        \norm{\bg}_{L^1(\Omega,\RR^d)} +
        \norm{f}_{L^{p'}(\Omega,\RR^d)}^2\right).
  \end{split}
\end{align}

It follows that $\uu^{\varepsilon} \in \mathrm{BV}_{\loc}(\Omega,\RR^d)$, see
\cite{attouchVariationalAnalysisSobolev2014}. Moreover, by
  choosing $\bg' = 0$ and $\ff = 0$ alternatively, then applying Riesz
theorem, we conclude
\begin{align}
\label{eq:10}
\norm{\uu^{\varepsilon}}_{\mathrm{BV}(\Omega',\RR^d)}
  &\le C (d,\kappa,\lambda_e,\Lambda_e,\lambda_o,\Lambda_o,\Omega',\Omega) \left( \norm{\bg}_{L^r(\Omega,\RR^d)}
    +  \norm{f}_{L^{p'}(\Omega,\RR^{d \times d})}^2\right).
\end{align}

\subsection{Higher regularity and existence}
\label{sec:normal-trace-elastic}
The a priori estimate \eqref{eq:10} shows that the solution
$\uu^{\varepsilon}$ of \eqref{eq:p410}-\eqref{eq:p415}, if it exists, is
merely a bounded variation function. However, this result does not
guarantee that the normal traces appearing in \eqref{eq:p413} are
well-defined, cf. e.g.,
\cite{chenTracesExtensionsBounded2020,phucCharacterizationsSignedMeasures2017,chenGaussGreenTheoremWeakly2009}. Therefore,
we will show next that $\uu^{\varepsilon}$ possesses higher
regularity. Indeed, for each $\varepsilon > 0$, the (unique) solution
$\uu^{\varepsilon}$ of \eqref{eq:p410}-\eqref{eq:p415} belongs to
$W_{\loc}^{1,q(\varepsilon)}(\Omega,\RR^d)$, for some
$1 < q(\varepsilon) < 2$ that will be given below.

To that end, we need to adapt a useful global gradient estimate,
obtained by N. C. Phuc \cite[Theorem
1.1]{phucGlobalIntegralGradient2014} to our current setting, as
follows:
\begin{proposition}
\label{sec:main-results-4}
Suppose $2-\frac{1}{d}<p \le d$ and $\bfa\colon Y \times \RR^d \to \RR^d$ satisfies
\ref{cond:a-periodic}--\ref{cond:a-monotonicity}. Let
$\mu \in L^s(\Omega)$ for some
$d > s >  \frac{1}{1 - \frac{1}{p}+\frac{1}{d}}.$ Then there
exists $\delta = \delta (d, p, \lambda_o,\Lambda_o) > 0$ such that for
any weak solution $\phi \in W_0^{1,p}(\Omega)$ of 
\begin{align}
  \label{eq:56}
  -\Div \left( \bfa (x, \nabla \phi) \right) = \mu \text{ in } \Omega,
  \qquad
  \phi = 0 \text{ on }\partial \Omega,
\end{align}
the following estimate holds:
\begin{align}
\label{eq:55}
  \int_{\Omega} \abs{\nabla \phi}^{p+\delta} \di x
  \le C\left( \int_{\Omega} \abs{\mu}^s \di x \right)^{\frac{p-1}{s(p+\delta)}}.
\end{align}
where $C = C (d, p, s, \lambda_o,\Lambda_o,\diam(\Omega))  >0.$
\end{proposition}

\begin{proof}[Proof of \cref{sec:main-results-4}]

  Since $s > \frac{1}{1 - \frac{1}{p}+\frac{1}{d}}$, there exists
$\delta_1 = \delta_1(p,d,s) > 0$ such that
$s > \frac{1}{ \frac{p-1}{p+\delta_1}+\frac{1}{d}} > \frac{1}{1 -
  \frac{1}{p}+\frac{1}{d}}$,
then
\begin{align}
  \label{eq:59}
\frac{p+\delta_1}{p-1} < \frac{ds}{d-s}.
\end{align}

  It is clear that if $\bfa\colon Y \times \RR^d \to \RR^d$ satisfies
  \ref{cond:a-periodic}--\ref{cond:a-monotonicity} then it also
  satisfies all the conditions in \cite[Theorem
  1.1]{phucGlobalIntegralGradient2014} (note that we can extend $\bfa$
  to $\tilde{\bfa}\colon \RR^d \times \RR^d \to \RR^d$ by
  periodicity). Therefore, we have
  $0< \delta = \delta(d,p,\lambda_o,\Lambda_o) < \delta_1$ small enough such
  that
\begin{align}
\label{eq:57}
  \int_{\Omega} \abs{\nabla \phi}^{p+\delta} \di x
  \le C \int_{\Omega} \calN_1 \left( \abs{\mu}
  \right)^{\frac{p+\delta}{p-1}} \di x,
\end{align}
for some $C = C (d,p,\delta,\diam(\Omega)) > 0,$ where $\calN_1$ is the fractional maximal function, defined as 
\begin{align*}
\calN_1 \left( \nu \right) (x) \coloneqq \sup_{r > 0} 
  \frac{r}{\abs{B_r(x)}} \nu \left( B_r(x) \right),
\end{align*}
for any nonnegative locally finite measure $\nu$ on $\RR^d$.

Since $0<\delta<\delta_1$ and $\Omega$ is bounded, by \eqref{eq:59} and H\"{o}lder's inequality, we have
\begin{align}
\label{eq:74}
\int_{\Omega} \calN_1 \left( \abs{\mu}
  \right)^{\frac{p+\delta}{p-1}} \di x
  \le C  \left( \int_{\Omega} \calN_1 \left( \abs{\mu}
  \right)^{\frac{ds}{d-s}} \di x \right)^{\frac{d-s}{ds}\cdot \frac{p-1}{p+\delta}}.
\end{align}

From \cite[Theorem
3.1]{kinnunenRegularityFractionalMaximal2003}, there exists $C =
C(d,p) > 0$ such that
\begin{align}
\label{eq:58}
  \int_{\Omega} \calN_1 \left( \abs{\mu} \right)^{\frac{ds}{d-s}}
  \le C \left( \int_{\Omega} \abs{\mu}^s \di x \right)^{\frac{d}{d-s}}
\end{align}

Combining \eqref{eq:57}, \eqref{eq:74}, and \eqref{eq:58}, we obtain \eqref{eq:55}.

\end{proof}

We are ready to show that the elastic displacement $\uu^{\varepsilon}$
actually belongs to some Sobolev spaces.

Fix $\varepsilon > 0$. Recall that
$\bg \in L^r (\Omega,\RR^d) \subset W^{-1,r}(\Omega, \RR^d)$ for some
$r > 1$. Applying \cref{sec:main-results-4} to the nonlinear
divergence problem \eqref{eq:p421}--\eqref{eq:p423}, and using the
fact that $\Omega$ is bounded, we obtain that there exists
$1 < q(\varepsilon) < \min \left\{ r,2 \right\}$ small enough, such
that
$\abs{\nabla \varphi^{\varepsilon}}^2 \in L^{q(\varepsilon)}(\Omega)$
and $\bg \in W^{-1,q(\varepsilon)}(\Omega,\RR^d)$. It also follows
that
$\maxss^{\varepsilon} \in L^{q(\varepsilon)}(\Omega, \RR^{d \times
  d})$.

Let $\bg'$ and $\ff'$ be some suitable Lebesgue integrable
  functions that will be specified later. Let
$\ww^{\varepsilon} \in W_0^{1,q(\varepsilon)'}(\Omega',\RR^d)$, where
$q(\varepsilon)^{'}$ is the H\"{o}lder conjugate of $q(\varepsilon)$,
be the (unique) solution of
\begin{align}
\label{eq:37}
\int_{\Omega} \bfB \left( \frac{x}{\varepsilon} \right) \, \DD(\ww^{\varepsilon}) : \DD(\uu)
  \di x 
  = \int_{\Omega} \bg' \cdot \uu \di x
  - \int_{\Omega} \ff' : \DD(\uu) \di x
\end{align}
for all $\uu \in \calD (\Omega,\RR^d)$.

For fixed $1 \le i, j \le d$, define the interpolation map
\begin{align*}
\mathrm{T}^{ij}\colon \left( L^2 (\Omega,\RR^d) \times L^2(\Omega,\RR^{d\times d}) \right)
  + \left( L^{\infty}(\Omega,\RR^d)\times L^{\infty}(\Omega,\RR^{d\times
  d}) \right)
  &\to L^2(\Omega') + L^{\infty} (
    \Omega')\\
  \left( \bg_2, \ff_2 \right) + \left( \bg_{\infty},\ff_{\infty}
  \right)
  &\mapsto \left[ \DD \ww^{\varepsilon}_2 + \DD \ww^{\varepsilon}_{\infty} \right]_{ij},
\end{align*}
where $\ww^{\varepsilon}_2$ and $\ww^{\varepsilon}_{\infty}$ are
solutions of \eqref{eq:37} with source terms
$(\bg',\ff') = (\bg_2,\ff_2)$ and
$(\bg',\ff') = (\bg_\infty,\ff_\infty)$, respectively.

On the one hand, the estimates \eqref{eq:29} and \eqref{eq:7} imply
the bounds of the restriction maps
\begin{align*}
&\norm{\mathrm{T}^{ij} \bigg\vert_{ L^2 (\Omega,\RR^d)\times L^2(\Omega,\RR^{d\times d})
  }}_{L^2 \to L^2}\\
  &\qquad\coloneqq \sup \left\{ \norm{\mathrm{T}^{ij}(\bg_2, \ff_2)}_{L^2(\Omega')}\colon
  \norm{\bg_2}_{L^2(\Omega,\RR^d)} \leq 1 , \norm{\ff_2}_{L^2(\Omega,\RR^d \RR^{d\times d} ) }
  \le 1 \right\}
  \le 2 C_2,\\
  &\norm{\mathrm{T}^{ij} \bigg\vert_{ L^\infty (\Omega,\RR^d)\times L^\infty(\Omega,\RR^{d\times d})
    }}_{L^\infty \to L^\infty}\\
  &\qquad\coloneqq \sup \left\{ \norm{T^{ij}(\bg_\infty, \ff_\infty)}_{L^{\infty}(\Omega')}\colon
  \norm{\bg_\infty}_{L^\infty(\Omega,\RR^d)} \leq 1 , \norm{\ff_\infty}_{L^\infty(\Omega,\RR^d \RR^{d\times d} )}
  \le 1 \right\}
  \le 2 C_1.
\end{align*}

On the other hand, the identification 
\begin{align*}
  \mathrm{Id} \colon L^s(\Omega,\RR^d) \times L^s (\Omega,\RR^{d\times d})
  &\to
  L^s (\Omega, \RR^{d^3})\\
  (\bg, \ff)
  &\mapsto
  (\bg_1, \ldots, \bg_d, \ff_{11}, \ldots, \ff_{1d}, \ldots,
  \ff_{d1},\ldots, \ff_{\,dd}) 
\end{align*}
is an isomorphism with respect $L^s-$norm for any $2 \le s \le \infty$.

Therefore, applying the Multilinear Riesz-Thorin Interpolation Theorem \cite[Corollary
7.2.11]{grafakosModernFourierAnalysis2014} to the map $\mathrm{T}^{ij} \circ \mathrm{Id}^{-1}$, there exists
$C_3(\varepsilon) > 0$ such that the norm of the restriction of $\mathrm{T}^{ij}$ on
\begin{align*}
  \left( L^{q (\varepsilon)} (\Omega,\RR^d)\times L^{q (\varepsilon)}(\Omega,\RR^{d\times d}) \right)
  \subset \left( L^2 (\Omega,\RR^d)\times L^2(\Omega,\RR^{d\times d}) \right)
  + \left( L^{\infty}(\Omega,\RR^d)\times L^{\infty}(\Omega,\RR^{d\times
  d}) \right)
\end{align*}
to $L^{q(\varepsilon)'}(\Omega')$ is bounded above by
$C_3(\varepsilon)$. In particular, the following $L^{{q
    (\varepsilon)'}}-$gradient estimate holds
\begin{align}
\label{eq:30}
\norm{\DD \ww^{\varepsilon}}_{L^{q (\varepsilon)'}(\Omega',\RR^{d\times
  d})}
  \le C_3(\varepsilon) d^2 \left( \norm{\bg'}_{L^{q
  (\varepsilon)'}(\Omega,\RR^d)} + \norm{\ff'}_{L^{q
  (\varepsilon)'}(\Omega,\RR^{d\times d})} \right)
\end{align}
whenever $\ww^{\varepsilon}$ is the solution of \eqref{eq:37} with
source term $(\bg',\ff')$ in $L^{q (\varepsilon)'}$.
By
\eqref{eq:30}, Korn's inequality, and H\"{o}lder's inequality, there
exists $C (\varepsilon) = C(\varepsilon, d,\kappa,\lambda_e,\Lambda_e,\lambda_o,\Lambda_o,\Omega',\Omega)$ such that 
\begin{align}
\label{eq:38}
  \norm{\ww^{\varepsilon}}_{W^{1,q(\varepsilon)'}(\Omega',\RR^d)}
  \le C(\varepsilon) \left(
  \norm{\bg'}_{L^{q(\varepsilon)'}(\Omega,\RR^d)} +
  \norm{\ff'}_{L^{q(\varepsilon)'}(\Omega,\RR^{d \times d})}\right). 
\end{align}

We now employ a duality argument to prove an a priori estimate for
$\uu^{\varepsilon} \in W_{\loc}^{1,q(\varepsilon)}(\Omega,\RR^d)
\subset \mathrm{BV}_{\loc}(\Omega,\RR^d)$ satisfying \eqref{eq:r7}.




Fix $\Omega' \subset\subset \Omega$. In \eqref{eq:37}, let $\bg' \in
\calD(\Omega',\RR^d),\ff \in \calD(\Omega',\RR^{d\times d})$ and
combine with \eqref{eq:r7}, \eqref{eq:38}, we obtain
\begin{align}
\label{eq:39}
  \begin{split}
    &\abs{\int_{\Omega'} \bg' \cdot \uu^{\varepsilon} \di x
  - \int_{\Omega'} \ff' : \DD(\uu^{\varepsilon}) \di x}\\
    &= \abs{\int_{\Omega'} \bfB \left( \frac{x}{\varepsilon} \right)\DD(\ww^{\varepsilon}) : \DD(\uu^{\varepsilon})
  \di x} \\
    &= \abs{\int_{\Omega'} \bg \cdot \ww^{\varepsilon} \di x
  - \int_{\Omega'} \bfC \left( \frac{x}{\varepsilon} \right)\maxss^{\varepsilon}: \DD(\ww^{\varepsilon})\di x}\\
    &\le C(\varepsilon) \left(
      \norm{\bg'}_{L^{q(\varepsilon)'}(\Omega',\RR^d)} +
      \norm{\ff'}_{L^{q(\varepsilon)'}(\Omega',\RR^{d\times
      d})}\right) \left(
      \norm{\bg}_{L^{q(\varepsilon)}(\Omega',\RR^d)} +
      \norm{\maxss^{\varepsilon}}_{L^{q(\varepsilon)}(\Omega',\RR^{d\times
      d})}\right)
  \end{split}
\end{align}

On the one hand, letting $\ff' = 0$ in \eqref{eq:39} implies 
\begin{align*}
  \abs{\int_{\Omega'} \bg' \cdot \uu^{\varepsilon} \di x}
  \le C(\varepsilon) 
      \norm{\bg'}_{L^{q(\varepsilon)'}(\Omega',\RR^d)} \left(
      \norm{\bg}_{L^{q(\varepsilon)}(\Omega',\RR^d)} +
      \norm{\maxss^{\varepsilon}}_{L^{q(\varepsilon)}(\Omega',\RR^{d\times
      d})}\right).
\end{align*}
Thus by Riesz Theorem, we obtain 
\begin{align*}
\norm{\uu^{\varepsilon}}_{L^{q(\varepsilon)}(\Omega',\RR^d)} \le
  C(\varepsilon) \left( \norm{\bg}_{L^{q(\varepsilon)}(\Omega',\RR^d)} +
  \norm{\maxss^{\varepsilon}}_{L^{q(\varepsilon)}(\Omega',\RR^{d\times
  d})} \right).
\end{align*}
On the other hand, letting $\bg' = 0$ in \eqref{eq:39} and arguing similarly, we obtain 
\begin{align*}
\norm{\DD(\uu^{\varepsilon})}_{L^{q(\varepsilon)}(\Omega',\RR^{d
  \times d})}
  \le C(\varepsilon) 
  \left(
      \norm{\bg}_{L^{q(\varepsilon)}(\Omega',\RR^d)} +
      \norm{\maxss^{\varepsilon}}_{L^{q(\varepsilon)}(\Omega',\RR^{d\times
      d})}\right).
\end{align*}
Combining the two estimates above, we conclude that
\begin{align}
\label{eq:40}
\norm{\uu^{\varepsilon}}_{W^{1,q(\varepsilon)}(\Omega',\RR^d)} \le
  C(\varepsilon) \left( \norm{\bg}_{L^{q(\varepsilon)}(\Omega',\RR^d)} +
  \norm{\maxss^{\varepsilon}}_{L^{q(\varepsilon)}(\Omega',\RR^{d\times
  d})} \right).
\end{align}

We have shown that every distributional solution $\uu^{\varepsilon}$
of \eqref{eq:r7} belongs to
$W_{\loc}^{1,q(\varepsilon)}(\Omega,\RR^d)$, and thus the normal
traces appearing in \eqref{eq:p413} are well-defined. It remains to
show that \eqref{eq:r7} has a distributional solution
$\uu^{\varepsilon} \in L^1_{\loc}(\Omega,\RR^d)$. This follows from
the existence and asymptotic decay of the Green's function associated with the
operator
$\calL^{\varepsilon}\coloneqq -\Div \left( \bfB
  \left( \frac{x}{\varepsilon}\right) \DD [\cdot]  \right) $ shown in \cite[Theorem 1]{conlonGreenFunctionElliptic2017}.  The uniqueness
of the distributional solution $\uu^{\varepsilon}$ follows from a
standard argument, by using density and the fundamental lemma of
calculus of variations.

\begin{remark}
\label{sec:high-regul-exist}
If we assume instead that $\bfB$ is in $VMO (\Omega,\RR^d)$, then the
$W^{1,p}-$estimate \eqref{eq:40} would be obtained via a real variable
method by Caffarelli and Peral
\cite{caffarelliW1PestimatesElliptic1998,shenPeriodicHomogenizationElliptic2018},
while the $W^{1,\infty}-$estimate \eqref{eq:29} was first obtained
via the compactness method
\cite{avellanedaCompactnessMethodsTheory1989,avellanedaCompactnessMethodsTheory1987}. In
this case, the existence of the solution $\uu^{\varepsilon}$ can be shown
by an approximation argument (the SOLA method - Solutions Obtained by
Limit of Approximations, see \cite{boccardoNonlinearEllipticParabolic1989}).


Indeed, from
\cite{caffarelliW1PestimatesElliptic1998,shenPeriodicHomogenizationElliptic2018},
we have the global estimate 
\begin{align}
\label{eq:93}
  \norm{\uu^{\varepsilon}}_{W^{1,q(\varepsilon)}(\Omega,\RR^d)}
  \le C \left( \norm{\bg}_{L^{q(\varepsilon)}(\Omega,\RR^d)} +
  \norm{\maxss}_{L^{q(\varepsilon)}(\Omega,\RR^{d\times d})} \right).
\end{align}
Let $\bg_n \in \calD(\Omega,\RR^d)$ and
$\ff_n \in \calD(\Omega,\RR^{d\times d})$ that converge to $\bg $ and
$\bfC \left( \frac{x}{\varepsilon} \right)\maxss^{\varepsilon}$,
respectively, in the $L^{q(\varepsilon)}-$norm, as $n\to \infty$. Observe that the
variational problem
\begin{align}
  \label{eq:41}
  \int_{\Omega} \bfB \left( \frac{x}{\varepsilon} \right)\DD(\uu_n^{\varepsilon}) : \DD(\vv)
  \di x 
  =\int_{\Omega} \bg_n\cdot \vv \di x
  - \int_{\Omega} \ff_n:\DD(\vv)\di x,
  \quad
  \forall \vv \in \calD (\Omega,\RR^d)
\end{align}
has a unique solution $\uu^{\varepsilon}_n \in
W_0^{1,2}(\Omega,\RR^d)\subset W_0^{1,q(\varepsilon)}(\Omega,\RR^d)$ by
the Lax-Milgram Theorem. Then, for $m,n \in \NN$, we have
$\uu^{\varepsilon}_m - \uu^{\varepsilon}_n$ is the solution of
\begin{align}
  \label{eq:42}
  \begin{split}
\int_{\Omega} \bfB \left( \frac{x}{\varepsilon} \right)\DD(\uu_m^{\varepsilon}-\uu_n^{\varepsilon}) : \DD(\vv)
  \di x
  = \int_{\Omega} \left( \bg_m - \bg_n \right) \cdot \vv \di x
  - \int_{\Omega} \left( \ff_m - \ff_n \right):\DD(\vv) \di x, \\
    \forall \vv \in \calD(\Omega,\RR^d).
  \end{split}
\end{align}
Because $\bg_m - \bg_n$ and $\ff_m - \ff_n$ are also in $L^{q(\varepsilon)}$, the
estimate \eqref{eq:40} applies to $\uu^{\varepsilon}_m -
\uu^{\varepsilon}_n$, so 
\begin{align}
\label{eq:43}
\norm{\uu^{\varepsilon}_m -\uu^{\varepsilon}_n}_{W^{1,q(\varepsilon)}(\Omega,\RR^d)} \le C(\varepsilon) \left(
  \norm{\bg_m - \bg_n}_{L^{q(\varepsilon)}(\Omega,\RR^d)} +
  \norm{\ff_m - \ff_n}_{L^{q(\varepsilon)}(\Omega,\RR^{d\times d})} \right) \cv 0
\end{align}
as $m, n \to \infty$. Therefore, $\uu^{\varepsilon}_n$ is a Cauchy
sequence in $W^{1,q(\varepsilon)}(\Omega,\RR^d)$, and so there
exists $\uu^{\varepsilon} \in W^{1,q(\varepsilon)}(\Omega,\RR^d)$
such that $\uu^{\varepsilon}_n\cv \uu^{\varepsilon}$, as
$n \to \infty$. By letting $n \to \infty$ in \eqref{eq:41} and using density arguments,  we obtain
$\uu^{\varepsilon}$ is the solution of \eqref{eq:r7}.
\end{remark}

\subsection{The two-scale homogenized system}
\label{sec:two-scale-homog}

Fix $K_0\subset\subset \Omega$, then there exists an open set $K$ such
that $K_0 \subset\subset K \subset\subset \Omega$.  By \eqref{eq:10},
there exists $\uu_{K_0}^0 \in \mathrm{BV}(K,\RR^d)$ and
$\uu_{K_0}^1 \in \calM \left( K, \mathrm{BV}_{\per}(Y,\RR^d) \right)$, such that,
up to a subsequence, see
\cite{amarTwoscaleConvergenceHomogenization1998},
\begin{subequations}
  \label{eq:12}
\begin{align}
\label{eq:11}
  \uu^{\varepsilon}
  &\wcv[\mathrm{BV}(K,\RR^d)] \uu_{K_0}^0,\\
  \nabla \uu^{\varepsilon}
  &\wcv[2] \nabla \uu_{K_0}^0 (\di x) + \nabla_y \uu_{K_0}^1 (\di x, \di y),
\end{align}
\end{subequations}
where $\calM \left( K, \mathrm{BV}_{\per}(Y,\RR^d)
\right)$ is the subspace of $\calM \left( K,
  L^{\frac{d}{d-1}}_{\per}(Y,\RR^d) \right)$ of measures $\mu$ from
the Borel $\sigma-$algebra on $K$ to $\mathrm{BV}_{\per}(Y,\RR^d)$ such that
$D_y \mu \in \calM \left( K \times Y, \RR^d \right) $, cf. \cite{amarTwoscaleConvergenceHomogenization1998}.
In \eqref{eq:r7}, letting $\vv(x) = \vv^0(x) + \varepsilon\vv^1 \left( x,
  \frac{x}{\varepsilon} \right)$ with $\vv^0 \in \calD \left(
  K,\RR^d \right)$ and $\vv^1 \in \calD \left( K,
  C_{\per}^{\infty}(Y) \right)$, we obtain
\begin{align*}
  &\int_{K} \bfB \left( \frac{x}{\varepsilon} \right) \left(
    \DD(\vv^0) + \varepsilon \, \DD \left( \vv^1 \right) \left( x,
    \frac{x}{\varepsilon} \right) + \DD_y \left( \vv^1 \right) \left( x,
    \frac{x}{\varepsilon}\right) \right): \DD \left( \uu^{\varepsilon}
    \right)\\
  &= \int_{K} \bg \cdot\left( \vv^0 + \varepsilon \, \vv^1 \right)
    + \int_{K} \bfC \left( \frac{x}{\varepsilon} \right) \maxss^{\varepsilon}(x) :\left( \DD
    \left( \vv^0 \right) + \varepsilon \, \DD \left( \vv^1 \right) + \DD_y
    \left( \vv^1 \right) \left( x,\frac{x}{\varepsilon} \right)\right)
\end{align*}

Taking $\varepsilon \to 0$, by \eqref{eq:12}  we have
\begin{align}
\label{eq:84}
  \begin{split}
    &\int_{K} \int_Y \bfB (y) \left( \DD \left( \vv^0 \right) +
      \DD_y \left( \vv^1 \right) \right) :\left( \DD \left( \uu_{K_0}^0
      \right)(\di x) + \DD_y \left( \uu_{K_0}^1 \right)(\di x, \di y)
      \right)\\
    &= \int_{K} \bg \cdot\vv^0 \di x
      + \lim_{\varepsilon \to 0} \int_{K} \bfC \left( \frac{x}{\varepsilon} \right) \maxss^{\varepsilon}(x) :\left( \DD
    \left( \vv^0 \right) + \varepsilon \, \DD \left( \vv^1 \right) + \DD_y
    \left( \vv^1 \right) \left( x,\frac{x}{\varepsilon}
      \right)\right)\di x\\
    &= \int_{K} \bg \cdot\vv^0 \di x
      + \lim_{\varepsilon \to 0} \int_{K} \bfC \left( \frac{x}{\varepsilon} \right) \maxss^{\varepsilon}(x) :\left( \DD
    \left( \vv^0 \right)+ \DD_y
    \left( \vv^1 \right) \left( x,\frac{x}{\varepsilon} \right)\right)
      \di x.
  \end{split}
\end{align}
To find the last limit, we claim that 
\begin{align}
\label{eq:85}
  \maxss^{\varepsilon}
  \wcv[2] \maxss^0(x,y)
  \coloneqq \left( \nabla \varphi^0 (x) + \nabla_y \varphi^1(x,y)
  \right) \otimes \left( \nabla \varphi^0 (x) + \nabla_y
  \varphi^1(x,y) \right) 
\end{align}
in $L^{p/2}(\Omega \times Y)$. Indeed, from \eqref{eq:112} and
\cref{sec:preliminary-results-1}, 
\begin{align*}
\frac{\partial \varphi^{\varepsilon}}{\partial x_i} \circ
  S^{\varepsilon}
  \cv \frac{\partial \varphi^0}{\partial x_i} + \frac{\partial
  \varphi^1}{\partial y_i} \quad \text{ in }L^p(\Omega \times Y).
\end{align*}
It follows that 
\begin{align*}
\left( \frac{\partial \varphi^{\varepsilon}}{\partial x_i} \frac{\partial \varphi^{\varepsilon}}{\partial x_j} \right)\circ
  S^{\varepsilon}
  \cv \left( \frac{\partial \varphi^0}{\partial x_i} + \frac{\partial
  \varphi^1}{\partial y_i} \right) \left( \frac{\partial \varphi^0}{\partial x_j} + \frac{\partial
  \varphi^1}{\partial y_j} \right) \quad \text{ in }L^{p/2}(\Omega \times Y),
\end{align*}
so by \cite[Proposition 2.5]{visintinTwoscaleCalculus2006} (this
result is necessary because it also applies to the case $p/2 = 1$),
\begin{align*}
 \frac{\partial \varphi^{\varepsilon}}{\partial x_i} \frac{\partial \varphi^{\varepsilon}}{\partial x_j} 
  \wcv[2] \left( \frac{\partial \varphi^0}{\partial x_i} + \frac{\partial
  \varphi^1}{\partial y_i} \right) \left( \frac{\partial \varphi^0}{\partial x_j} + \frac{\partial
  \varphi^1}{\partial y_j} \right) \quad \text{ in }L^{p/2}(\Omega \times Y).
\end{align*}
This convergence implies \eqref{eq:85}.
Since $\bfC \left( y \right) \left( \DD
    \left( \vv^0 \right)(x) + \DD_y
    \left( \vv^1 \right) \left( x,y
    \right)\right)$ is continuous with respect to $x$ and measurable
  with respect to $y$, it can be chosen as the test function for
  \eqref{eq:85}. Therefore, \eqref{eq:84} becomes
\begin{align}
\label{eq:14}
  \begin{split}
    &\int_{K} \int_Y \bfB (y) \left( \DD \left( \vv^0 \right) +
      \DD_y \left( \vv^1 \right) \right) :\left( \DD \left( \uu_{K_0}^0
      \right)(\di x) + \DD_y \left( \uu_{K_0}^1 \right)(\di x, \di y)
      \right)\\
    &= \int_{K} \bg \cdot\vv^0 \di x + \int_{K} \int_Y \bfC(y) \maxss^0(x,y): \left( \DD
      (\vv^0) + \DD_y (\vv^1) \right)  \di x \di y.
  \end{split}
\end{align}

Applying \cite[Theorem 1]{meyersResultsRegularitySolutions1975}  to the cell problem \eqref{eq:28}
and  \cref{sec:main-results-4} to the first homogenized equation in \eqref{summary-eqn}, there
exists $q^+ > 1$ such that
$\maxss^0 \in L^{q^+}(\Omega \times Y,\RR^{d\times d}).$ In
\eqref{eq:14}, setting $\vv^0 \equiv 0$,
$\vv^1 (x,y) = \eta (x)\ww(y)$ for
$\eta \in \calD(K),~\ww \in C^{\infty}_{\per}(Y,\RR^d)$, then use
the fundamental lemma of calculus of variation to obtain 
\begin{align*}
  \int_Y \bfB (y) 
  \DD_y \left( \ww^1 \right) :\left( \DD \left( \uu_{K_0}^0
  \right)(\di x) + \DD_y \left( \uu_{K_0}^1 \right) (\di y)
  \right)
  &=  \int_Y \bfC(y) \maxss^0(x,y):  \DD_y (\ww^1)   \di y,
\end{align*}
for a.e. $x \in K$. Now repeating the interpolation and duality
argument presented in \cref{sec:normal-trace-elastic}, we have
$ \DD \left( \uu_{K_0}^0 \right) + \DD_y \left( \uu_{K_0}^1 \right) $
is indeed in $L^{q^+}(Y,\RR^{d \times d})$, for a.e. $x \in K$. Since
$\uu_{K_0}^0$ depends only on $x$, we conclude
$\DD_y \left( \uu_{K_0}^1 \right) $ is in
$L^{q^+}(Y,\RR^{d \times d})$, for a.e. $x \in K$. Similarly, if we
set $\vv^1 \equiv 0$, $\vv^0 \in \calD(K,\RR^d),$ we obtain
$\DD \left( \uu_{K_0}^0 \right) + \DD_y \left( \uu_{K_0}^1 \right)$ is
in $L^{q^+}(K,\RR^d)$ for a.e. $y \in Y$. Integrating the sum over
$Y$, the last term vanishes due to periodicity, and therefore,
$\DD \left( \uu_{K_0}^0 \right) \in L^{q^+}(K,\RR^d)$, hence,
$\uu_{K_0}^0 \in W^{1,q^+}(K_0,\RR^d)$. As a consequence,
\eqref{eq:14} can be written in the classical form, without any Radon
measures, i.e.,
\begin{align}
\label{eq:94}
  \begin{split}
    &\int_{K} \int_Y \bfB (y) \left( \DD \left( \vv^0 \right) +
      \DD_y \left( \vv^1 \right) \right) :\left( \DD \left( \uu_{K_0}^0
      \right) + \DD_y \left( \uu_{K_0}^1 \right)\right) \di x \di y
      \\
    &= \int_{K} \bg \cdot\vv^0 \di x + \int_{K} \int_Y \bfC(y) \maxss^0(x,y): \left( \DD
      (\vv^0) + \DD_y (\vv^1) \right)  \di x \di y.
  \end{split}
\end{align}

Notice that the interpolation and duality arguments also provide
\begin{align}
\label{eq:96}
\norm{ \DD \left( \uu_{K_0}^0
  \right) + \DD_y \left( \uu_{K_0}^1 \right)}_{L^{q^+}(K_0 \times
  Y,\RR^{d \times d})}
  \le C \left( \norm{\bg}_{L^{q^+}(K_0,\RR^d)} +
  \norm{\maxss^0}_{L^{q^+}(K_0 \times Y,\RR^{d\times d})} \right).
\end{align}
Thus, by the SOLA argument used in \cref{sec:high-regul-exist}, we conclude
that \eqref{eq:94} has a unique (up to a constant) solution $(\uu_{K_0}^0,\uu_{K_0}^1)
\in W^{1,q^+}(K_0,\RR^d) \times L^{q^+} \left( K_0,
  W^{1,q^+}_{\per}(Y,\RR^d) \right) $.

For any $ x \in \Omega$, set
\begin{align}
\label{eq:102}
  (\ff^0,\ff^1)
  \coloneqq
  (\DD(\uu^0_{K_0}),\DD_y(\uu^1_{K_0})) \qquad \text{ if } x \in K_0 \text{ for
  some } K_0 \subset\subset \Omega,
\end{align}
then $(\ff^0,\ff^1)$ is well-defined due to the uniqueness of solution
of \eqref{eq:94}.
Let $(\uu^0,\uu^1) \in W_0^{1,q^+}(\Omega,\RR^d) \times L^{q^+} \left( \Omega, W^{1,q^+}_{\per}(Y,\RR^d) \right)$ be the solution of 
\begin{align}
\label{eq:103}
  &\int_{\Omega} \int_Y \left( \DD \left( \vv^0 \right) +
      \DD_y \left( \vv^1 \right) \right) :\left( \DD \left( \uu^0
      \right) + \DD_y \left( \uu^1 \right)\right) \di x \di y
      \\
    &= \int_{\Omega} \int_Y \left( \ff^0 + \ff^1 \right): \left( \DD
      (\vv^0) + \DD_y (\vv^1) \right)  \di x \di y.
\end{align}
then such solution exists and is unique by the SOLA argument used in
\cref{sec:high-regul-exist}. Moreover, 
\begin{align}
\label{eq:95}
  (\DD(\uu^0),\DD_y(\uu^1))
  &=(\DD(\uu_{K_0}^0),\DD(\uu_{K_0}^1)) &&\text{ if } x \in K_0 \text{ for some }
                                K_0\subset\subset \Omega,\\
  (\uu^0,\uu^1)
  &=0 &&\text{ if } x \in \partial \Omega,
\end{align}
It follows that $\uu^{\varepsilon} \wcv \uu^0$ in
distribution and by \eqref{eq:94},
\begin{align}
\label{eq:97}
  \begin{split}
    &\int_{\Omega} \int_Y \bfB (y) \left( \DD \left( \vv^0 \right) +
      \DD_y \left( \vv^1 \right) \right) :\left( \DD \left( \uu^0
      \right) + \DD_y \left( \uu^1 \right)\right) \di x \di y
      \\
    &= \int_{\Omega} \bg \cdot\vv^0 \di x + \int_{\Omega} \int_Y \bfC(y) \maxss^0(x,y): \left( \DD
      (\vv^0) + \DD_y (\vv^1) \right)  \di x \di y.
  \end{split}
\end{align}

The cell problems \eqref{eq:504}--\eqref{eq:504b} and the effective
system \eqref{summary-eqn} are derived by substituting the ansatz
  \begin{align*}
  \uu^1(x, y)
  &\coloneqq -\DD\left(\uu^0(x)\right)_{ij} \vec{\Upsilon}^{ij}(y) +
  \frac{\partial \varphi^0}{\partial x_i}(x) \frac{\partial
  \varphi^0}{\partial x_j}(x) \vec{\chi}^{ij}(y).
\end{align*}
The derivation of the effective coupled equation in \eqref{summary-eqn} follows in a similar fashion to \cite[Section
3.4]{dangHomogenizationNondiluteSuspension2021}.

Finally, the properties of $\bfa^{\hom}$ and $\bfB^{\hom}$ follow from
similar arguments as the ones presented in
\cite{chiadopiatHomogenizationMonotoneOperators1990,chiadopiatConvergenceMonotoneOperators1990} and
\cite{dangHomogenizationNondiluteSuspension2021}, and will be omitted here.

\hfill$\Box$

\section{Conclusions}
\label{sec:conclusions}
This paper is devoted to the periodic homogenization of nonlinear electric elastomers. More specifically, the nonlinear system \eqref{eq:p410}-\eqref{eq:p415} of an electrostatic equation coupled with an elasticity equation with periodic, highly oscillatory coefficients is considered, where $\varepsilon \ll 1$ is the size of microstructure.  
It is shown that the effective response of this system, given by \eqref{summary-eqn} with \eqref{eq:513}, consists of the homogeneous dielectric elastomer described by a nonlinear weakly coupled system of PDEs, whose coefficients depend on the coefficients of the original heterogeneous material and the geometry of the composite and the periodicity of the original microstructure. In particular, the effective coefficients 
\eqref{eq:513} are written in terms of solutions to the cell problems \eqref{eq:13}, \eqref{eq:504}, \eqref{eq:504b}. The main homogenization result is given in \cref{thm:main}, and the explicit corrector for the solution to the electrostatic problem is presented in \cref{sec:main-results-1}. It is worth noticing that in most of the existing literature on the topic of homogenization of monotone operators \cite{dalmasoCorrectorsHomogenizationMonotone1990,jimenezCorrectorsFieldFluctuations2010,bystromCorrectorsNonlinearMonotone2001}, the corrector results were not explicit. However, unlike the ones that did obtain the explicit corrector, e.g. \cite{allaireHomogenizationTwoscaleConvergence1992}, our results were obtained under minimal regularity assumptions. 
The {\it linear} case of 
\cite{francfortEnhancementElastodielectricsHomogenization2021,tianDielectricElastomerComposites2012} for $p=2$ can be recovered as a particular case of our analysis, that could also be straightforwardly extended to the case
when $\bfB$ is a $VMO$-function.
In addition, this paper contains two $L^p$-gradient estimates for elastic systems with discontinuous coefficients (see \cref{sec:lp-grad-estim} for \cref{sec:appendix-1} and \cref{sec:appendix}) that, together with the main result of the paper \cref{sec:main-results-1} (and its
extensions in \cref{sec:some-inter-corr}), constitute three stand alone results that could be found useful for the homogenization of electrostatic and/or elastic equations.

\appendix

\section{Appendix}
\label{sec:an-appendix}

\subsection{Other  first order correctors}
\label{sec:some-inter-corr}
Our arguments in the proof of \cref{sec:main-results-1} can be
modified to recover the first order correctors presented in previous
studies, e.g
\cite{bystromCorrectorsNonlinearMonotone2001,jimenezCorrectorsFieldFluctuations2010}.
In particular, the strong two-scale convergence and the explicit formula
obtained below by adapting our arguments are new.  To illustrate, we will
state here two results without proofs.

Our first proposition improves the one obtained in
\cite{bystromCorrectorsNonlinearMonotone2001}:
\begin{proposition}
\label{sec:other-first-order}
Let $c_1,  c_2 > 0$, $1 < p < \infty$, $0 \le \alpha \le \min \left\{
  1, p-1 \right\},$ and $\max \left\{ p,2 \right\} \le \beta <
\infty$. Let $\bfa: \Omega \times Y \times \RR^d \to  \RR^d$ be a
function such that $\bfa \left( x, \cdot, \xi \right)$ is measurable
and $Y-$periodic for $x \in \Omega$ and $\xi \in \RR^d$. Suppose
further that for $x \in \Omega,$ $y \in Y$, and $\xi_1, \xi_2 \in
\RR^d$, we have 
\begin{align}
\label{eq:24}
  \abs{\bfa(x,y,\xi_1) - \bfa (x,y,\xi_2)}
  &\le c_1 \left( 1 + \abs{\xi}_1 + \abs{\xi_2} \right)^{p-1-\alpha}
    \abs{\xi_1 - \xi_2}^{\alpha},\\
  \left( \bfa(x,y,\xi_1) - \bfa (x,y,\xi_2) \right)\cdot \left( \xi_1
  - \xi_2 \right)
  &\ge c_2 \left( 1+ \abs{\xi_1} + \abs{\xi_2} \right)^{p-\beta}
    \abs{\xi_1 - \xi_2}^{\beta}, \quad \text{ and }\\
  \bfa (x,y,0)
  &= 0.
\end{align}
Then for $f \in W^{-1,p'}(\Omega)$, there exist $\varphi^0 \in
W_0^{1,p}(\Omega)$ and $\varphi^1 \in L^p \left( \Omega,
  W_{\per}^{1,p}(Y)/\RR \right)$ such that the (unique) solution
$\varphi^{\varepsilon} \in W_0^{1,p}(\Omega)$ of 
\begin{align}
  \label{eq:25}
-\Div \left( \bfa \left( x, \frac{x}{\varepsilon}, \nabla
  \varphi^{\varepsilon} \right) \right)
  =f
\end{align}
satisfies 
\begin{align}
\label{eq:5}
  \varphi^{\varepsilon} \wcv[2] \varphi^0,
  \qquad\frac{\partial \varphi^{\varepsilon}}{\partial x_i}
  \cv[2] \frac{\partial \varphi^0}{\partial x_i} + \frac{\partial
  \varphi^1}{\partial y_i}.
\end{align}
In particular, if $\varphi^1$ is admissible, then
\begin{align}\label{eq:47}
\lim_{\varepsilon \to 0}\norm{\nabla \varphi^{\varepsilon}(\cdot) - \nabla \varphi^0(\cdot) -
  \nabla_y \varphi^1  \left( \cdot,\frac{\cdot}{\varepsilon}
  \right)}_{L^p(\Omega,\RR^{d})}
  = 0.
\end{align}
\end{proposition}

Next, we provide an explicit corrector for the one obtained in \cite{jimenezCorrectorsFieldFluctuations2010}: 
\begin{proposition}
  \label{sec:other-first-order-1}
  Suppose $F$ is a proper subset of the unit cell $Y$ with smooth
  boundary. Let 
\begin{align*}
  \sigma(y)
  &\coloneqq \charfn_F(y) \sigma_1 + (1-\charfn_F(y)) \sigma_2,\\
  p(y)
  &\coloneqq \charfn_F(y) p_1 + (1-\charfn_F(y)) p_2,
\end{align*}
for some fixed constant $\sigma_1, \sigma_2 > 0$, and $2 \le p_1 \le
p_2$. Define 
\begin{align*}
  \bfa (y, \xi)
  \coloneqq \sigma(y) \abs{\xi}^{p(y)-2}\xi, \qquad \text{ for } y \in
  Y, ~\xi \in \RR^d.
\end{align*}

Then for $f \in W^{-1,p'_1}(\Omega)$, there exist $\varphi^0 \in
W_0^{1,p_1}(\Omega)$ and $\varphi^1 \in L^{p_1} \left( \Omega,
  W_{\per}^{1,p_1}(Y)/\RR \right)$ such that the (unique) solution
$\varphi^{\varepsilon} \in W_0^{1,p_1}(\Omega)$ of 
\begin{align}
  \label{eq:100}
-\Div \left( \bfa \left( \frac{x}{\varepsilon}, \nabla
  \varphi^{\varepsilon} \right) \right)
  =f
\end{align}
satisfies 
\begin{align}
  \label{eq:83}
  \varphi^{\varepsilon} \wcv[2] \varphi^0, \qquad\frac{\partial \varphi^{\varepsilon}}{\partial x_i}
  \cv[2] \frac{\partial \varphi^0}{\partial x_i} + \frac{\partial
  \varphi^1}{\partial y_i}.
\end{align}
In particular, if $\varphi^1$ is admissible, then
\begin{align}\label{eq:99}
\lim_{\varepsilon \to 0}\norm{\nabla \varphi^{\varepsilon}(\cdot) - \nabla \varphi^0(\cdot) -
  \nabla_y \varphi^1  \left( \cdot,\frac{\cdot}{\varepsilon}
  \right)}_{L^p(\Omega,\RR^d)}
  = 0.
\end{align}
\end{proposition}

\subsection{$L^p$-gradient estimates in elasticity}
\label{sec:lp-grad-estim}

The argument in \cref{sec:normal-trace-elastic} leads to two
$W^{1,r}$-estimates in elasticity, which may be useful for the study
of elastic systems.

\begin{proposition}
\label{sec:appendix-1}
Suppose the $C^{1,\alpha}$-domain $\Omega$ is a disjoint union of a finite
$N_{\Omega}$  subdomains with piecewise $C^{1,\alpha}$-boundaries, where 
$\alpha \in (0,1)$. Let
$\bfB \in \fraM \left( \lambda_e,\Lambda_e \right)$ be H\"{o}lder
continuous on the closure of each subdomain. Suppose
$\bg \in L^r(\Omega,\RR^d)$ for $1\leq r\leq \infty$. Then the weak
solution $\uu \in W_0^{1,r}(\Omega,\RR^d)$ of
\begin{align}
\label{eq:45}
 -\Div \left[ \bfB  \nabla
  \uu \right]
  &= \bg \quad \text{ in } \Omega,
\end{align}
satisfies the following estimate: \\
For $\Omega' \subset\subset \Omega$, there exists $C = C(d, \alpha, N_{\Omega}, \lambda_e,\Lambda_e,r, \Omega',\Omega) > 0$ such that
\begin{align}
\label{eq:46}
  \norm{\uu}_{W^{1,r}(\Omega',\RR^d)}
  \le  C \norm{\bg}_{L^r(\Omega,\RR^d)}.
\end{align}
\end{proposition}

In the context of homogenization, we have the following result:
\begin{proposition}
\label{sec:appendix}
Let $\bfB\in \fraM_{\mat} (\lambda_e,\Lambda_e)$ and $\Omega, \alpha$ be as above. Assume further that the unit cell $Y$ is a disjoint union of
$N_Y$ finite subdomains with piecewise $C^{1,\alpha}-$boundaries and
$\bfB$ is H\"{o}lder continuous on the closure of each
subdomain. Suppose $\bg \in L^r(\Omega,\RR^d)$ for
$1\leq r\leq \infty$. Then the weak solution
$\uu^{\varepsilon} \in W_0^{1,r}(\Omega,\RR^d)$ of
\begin{align}
\label{eq:45}
 -\Div \left[ \bfB \left( \frac{x}{\varepsilon} \right) \nabla
  \uu^{\varepsilon} \right]
  &= \bg \quad \text{ in } \Omega,
\end{align}
satisfies the following estimate: \\
For $\Omega' \subset\subset \Omega$, there exists $C = C (d, \alpha, N_Y, \lambda_e,\Lambda_e,r, \Omega',\Omega) > 0$,
independent of $\varepsilon > 0$, such that 
\begin{align}
\label{eq:46}
  \norm{\uu^{\varepsilon}}_{W^{1,r}(\Omega',\RR^d)}
  \le  C \norm{\bg}_{L^r(\Omega,\RR^d)}.
\end{align}
\end{proposition}
This proposition holds because the constant $C_1$ in \eqref{eq:7} can
be improved, so that it does not depend on $\varepsilon$, thanks to the
celebrated \emph{compactness method},
cf. e.g. \cite{avellanedaCompactnessMethodsTheory1989,avellanedaCompactnessMethodsTheory1987,liEstimatesEllipticSystems2003,dangGlobalGradientEstimate2022,prangeUniformEstimatesHomogenization2014,shenPeriodicHomogenizationElliptic2018,kenigHomogenizationEllipticSystems2013,armstrongQuantitativeStochasticHomogenization2019},
which exploits additional regularity information: \eqref{eq:45} is
$H$-convergent to a system with constant coefficients. Similar $W^{1,r}$-estimates were proved for the case when
$\bfB$ is in $VMO$ or $BMO$ spaces
cf. e.g., \cite{caffarelliW1PestimatesElliptic1998,shenPeriodicHomogenizationElliptic2018,armstrongQuantitativeStochasticHomogenization2019}
and references therein.

\section{Acknowledgements}
The work of the third author was supported  by NSF grant DMS-2110036.  This
material is based upon work supported by and while serving at the
National Science Foundation for the second author Yuliya Gorb. Any
opinion, findings, and conclusions or recommendations expressed in
this material are those of the authors and do not necessarily reflect
views of the National Science Foundation.

\bibliographystyle{habbrv}
\bibliography{homogenisation,arxiv,topological-insulator}

\end{document}